\documentclass{article}
\usepackage{booktabs}
\usepackage{multirow}
\usepackage{amsmath}
\usepackage{geometry}
\usepackage{amsfonts}
\usepackage{amsthm}
\usepackage{amssymb}
\usepackage[dvipsnames]{xcolor}
\usepackage[shortlabels]{enumitem}
\usepackage{tikz-cd}

\usepackage{hyperref}
\usepackage{cleveref}

\usepackage[normalem]{ulem}

\title{On $7$-adic Galois representations for elliptic curves over $\Q$}
\author{Lorenzo Furio and Davide Lombardo}
\date{}

\newtheorem{theorem}{Theorem}
\newtheorem{conjecture}[theorem]{Conjecture}
\newtheorem{corollary}[theorem]{Corollary}

\newtheorem{lemma}[theorem]{Lemma}

\newtheorem{proposition}[theorem]{Proposition}

\theoremstyle{definition}
\newtheorem{definition}[theorem]{Definition}
\newtheorem{remark}[theorem]{Remark}

\newtheorem{question}[theorem]{Question}
\newtheorem{program}[theorem]{Program}

\newcommand{\N}{\mathbb{N}}
\newcommand{\Z}{\mathbb{Z}}
\newcommand{\Q}{\mathbb{Q}}

\newcommand{\F}{\mathbb{F}}
\newcommand{\OK}{\mathcal{O}_K}
\newcommand{\Gal}{\operatorname{Gal}}

\newcommand{\Aut}{\operatorname{Aut}}
\newcommand{\GL}{\operatorname{GL}}

\newcommand{\pRep}{\operatorname{Im}\rho_{E,p}}
\newcommand{\pnRep}[1]{\operatorname{Im}\rho_{E,p^#1}}

\numberwithin{theorem}{section}

\begin{document}

\maketitle

\begin{abstract}
    In recent years, significant progress has been made on Mazur's Program B, with many authors beginning a systematic classification of all possible images of $p$-adic Galois representations attached to elliptic curves over $\Q$. Currently, the classification is only complete for $p \in \{2,3,13,17\}$. The main difficulty for other primes arises from the need to understand elliptic curves whose mod-$p^n$ Galois representations are contained in the normaliser of a non-split Cartan subgroup. Equivalently, this amounts to determining the rational points on the modular curves $X_{ns}^+(p^n)$.
    Here, we consider the case $p=7$ and show that the modular curve $X_{ns}^+(49)$, of genus 69, has no non-CM rational points.
    To achieve this, we establish a correspondence between the rational points on $X_{ns}^+(49)$ and the primitive integer solutions of the generalised Fermat equation $a^2 + 28b^3 = 27 c^7$, the resolution of which can be reduced to determining the rational points of several genus-three curves. Furthermore, we reduce the complete classification of $7$-adic images to the determination of the rational points of a single plane quartic.
\end{abstract}

\section{Introduction}

Let $K$ be a number field and let $E$ be an elliptic curve defined over $K$. For every positive integer $N$, the action of the absolute Galois group of $K$ on the $N$-torsion points of $E$ defines a representation
\begin{equation*}
	\rho_{E,N}: \Gal(\overline{K}/K) \to \Aut(E[N]) \cong \GL_2\left({\Z}/{N\Z}\right).
\end{equation*}
For a fixed prime $p$, we can restrict to values of $N$ of the form $p^n$ and take the limit over $n$. We then obtain the $p$-adic representation
\begin{equation*}
	\rho_{E,p^\infty}: \Gal(\overline{K}/K) \to \Aut(T_pE) \cong \GL_2(\Z_p),
\end{equation*}
where $T_pE = \varprojlim E[p^n]$ is the $p$-adic Tate module of $E$.
In 1972, Serre \cite{serre72} proved his celebrated open image theorem, which states that if the elliptic curve $E$ does not have (potential) complex multiplication, then the representation $\rho_{E,p^\infty}$ is surjective for almost all primes $p$. He actually proved a stronger statement: if $E$ does not have complex multiplication, the image of the  \textit{adelic representation}
\begin{equation*}
	\rho_E = \prod_{p \text{ prime}} \rho_{E,p^\infty} : \Gal(\overline{K}/K) \to \prod_{p \text{ prime}} \Aut(T_pE) \cong \GL_2(\widehat{\Z})
\end{equation*}
is open; equivalently, the image of $\rho_E$ has finite index in $\GL_2(\widehat{\Z})$.
In the same paper, Serre asked the following question.

\begin{question}[Serre's uniformity question]\label{question: uniformity for Q, mod-p version}
	Let $K$ be a number field. Does there exist a constant $N$, depending only on $K$, such that for every non-CM elliptic curve ${E}/{K}$ and for every prime $p>N$ the residual representation $$\rho_{E,p} : \Gal\left({\overline{K}}/{K}\right) \to \Aut(E[p]) \cong \operatorname{GL}_2(\F_p)$$ is surjective?
\end{question}

Although this problem has been extensively studied, the question remains open even in the case $K=\Q$. In this latter case, it is conjectured that the answer is affirmative, with the best possible constant being $N=37$.

Over the years, many partial results have been proven towards resolving Question \ref{question: uniformity for Q, mod-p version}. When the representation $\rho_{E,p}$ is not surjective, its image must be contained in a maximal subgroup of $\GL_2(\F_p)$. The classification of maximal subgroups of $\GL_2(\F_p)$ shows that they are of three types: some so-called `exceptional' subgroups, the Borel subgroups, and the normalisers of (split or non-split) Cartan subgroups.
Serre showed \cite[§8.4, Lemma 18]{serre81} that the image of $\rho_{E, p}$ cannot be contained in an exceptional subgroup for $p>13$.
Mazur \cite{mazur78} later proved that there are no isogenies of prime degree $p$ between non-CM elliptic curves over $\mathbb{Q}$ for $p>37$: this is equivalent to the statement that for $p>37$ the image of $\rho_{E, p}$ is not contained in a Borel subgroup. 
More recently, Bilu and Parent developed their version of Runge's method for modular curves \cite{bilu11runge}, which allowed them to prove that $\operatorname{Im}\rho_{E,p}$ is not contained in the normaliser of a split Cartan for sufficiently large $p$ \cite{bilu11split}. The result was then sharpened by Bilu-Parent-Rebolledo \cite{bilu13}, who showed that the same statement holds for every $p \geq 11$, with the possible exception of $p=13$.
Finally, the result was extended to also cover the prime $p=13$ by means of the so-called \textit{quadratic Chabauty} method \cite{balakrishnan19}.

Thus, for $p>37$, the only remaining possibility is that $\operatorname{Im} \rho_{E,p}$ is contained in the normaliser of the non-split Cartan, for which we only have partial results \cite{zywina15, lefournlemos21, furiolombardo23}. We will denote this group by $C_{ns}^+(p)$ (see Definitions \ref{def: p-adic groups} and \ref{def:cartan}).

A few years after Serre's foundational article \cite{serre72}, in the same spirit as Question \ref{question: uniformity for Q, mod-p version}, Mazur proposed his famous ``Program B''.

\begin{program}[Mazur's Program B]\label{program: B}
    Given a number field $K$ and a subgroup $G$ of $\GL_2(\widehat{\Z})$, classify all elliptic curves $E/K$ whose associated adelic Galois representation maps $\Gal(\overline{K}/K)$ into $G$.
\end{program}

If Question \ref{question: uniformity for Q, mod-p version} has an affirmative answer, the next natural step towards Mazur's Program B is to classify all the possible $p$-adic images for primes $p \le N$.
In the case $K=\Q$, substantial advances in this direction were made by Rouse--Zureick-Brown \cite{rzb15}, Sutherland--Zywina \cite{sutherlandzywina17}, and Rouse--Sutherland--Zureick-Brown \cite{rszb22}, who conducted an extensive study of modular curves of small level. Indeed, it is well known that elliptic curves with prescribed Galois representations on torsion points correspond to rational points on modular curves.
Their results are summarised in \cite[Theorem 1.6]{rszb22} as follows.

For a prime $p$ and an open subgroup $G < \GL_2(\Z_p)$, denote by $X_G$ the corresponding modular curve in the following sense: the group $G$ has finite index in $\GL_2(\Z_p)$, hence there exists a (minimal) positive integer $n$ such that $G \supset I + p^nM_2(\Z_p)$. Setting $\overline{G} = G \pmod {p^n}$ we have $\frac{\GL_2(\Z_p)}{G} \cong \frac{\GL_2(\Z/p^n\Z)}{\overline{G}}$; we then define $X_G := X_{\overline{G}}$.
We will refer to the modular curves classified in \cite{rszb22} by the labels given there, which we call \textit{RSZB labels}.

\begin{theorem}[Rouse--Sutherland--Zureick-Brown]\label{thm: rszb}
    Let $p$ be a prime, let $E/\Q$ be an elliptic curve without complex multiplication, and let $G:=\operatorname{Im}\rho_{E,p^\infty}$. Exactly one of the following is true:
    \begin{enumerate}[(i)]
        \item the modular curve $X_G$ is isomorphic to $\mathbb{P}^1$ or an elliptic curve of rank one, in which case $G$ is identified in \cite[Corollary 1.6]{sutherlandzywina17} and $\langle G, -I \rangle$ is listed in \cite[Tables 1--4]{sutherlandzywina17};
        \item the modular curve $X_G$ has finitely many non-CM rational points, and $G$ is listed in \cite[Table 1]{rszb22} together with the non-CM rational points of $X_G$ (expressed as $j$-invariants or models of the corresponding elliptic curves);
        \item $G \pmod {p^n}$ is contained in $C_{ns}^+(p^n)$ for some $p^n \in \{3^3, 5^2, 7^2, 11^2\} \cup \{p \text{ prime} \mid p \ge 19\}$, with $C_{ns}^+(p^n)$ as in Definitions \ref{def: p-adic groups} and \ref{def:cartan};
        \item $G$ is a subgroup of one of the groups with RSZB label \verb|49.147.9.1| or \verb|49.196.9.1|.
    \end{enumerate}
\end{theorem}

Recently, the rational points on the modular curve $X_{ns}^+(27)$ have been determined by means of the quadratic Chabauty method \cite{balakrishnan25}, and hence the $3$-adic Galois representations attached to elliptic curves over $\Q$ have been completely classified.

In this article, we study $7$-adic representations and attempt to complete their classification. As \Cref{thm: rszb} shows, this is equivalent to determining the rational points on the modular curves $X_{ns}^+(49)$, \verb|49.147.9.1|, and \verb|49.196.9.1|.
From now on, we will use the notation $X_{ns}^\#(49)$ and $X_{sp}^\#(49)$ to denote the modular curves with RSZB labels \verb|49.147.9.1| and \verb|49.196.9.1|, respectively.
Both of these curves have genus $9$, while the curve $X_{ns}^+(49)$ has genus $69$. As a consequence, determining the rational points on $X_{ns}^+(49)$ appears \textit{a priori} to be the major obstacle in classifying $7$-adic representations. Nevertheless, we can prove the following theorem, which is the main result of this paper.

\begin{theorem}\label{thm: Cns+(49)}
    The modular curve $X_{ns}^+(49)$ has exactly $7$ rational points, all of which are CM.
\end{theorem}

Although this doesn't complete the classification of $7$-adic images, we note that it has the following useful consequence.
\begin{corollary}
    For every non-CM elliptic curve $E/\Q$, the image of the $7$-adic representation is the inverse image in $\GL_2(\Z_7)$ of $\operatorname{Im} \rho_{E, 49} \subseteq \operatorname{GL}_2(\Z/49\Z)$.
\end{corollary}
\begin{proof}
    Follows immediately from \Cref{thm: Cns+(49)}, \cite[Theorem 1.9]{furio24}, and \cite[Proposition 1]{bourdon25}.
    Alternatively, the result also follows by combining \Cref{thm: Cns+(49)}, \cite[Theorem 3.14]{furio24}, the classification in \cite[Theorem 1.5]{zywina15}, and the fact that $\frac{|\operatorname{Im}\rho_{E,49}|}{|\operatorname{Im}\rho_{E,7}|} \ge 7^2$. This last property follows from \cite[Lemma 6.2]{furio24} in the non-split Cartan case, and can be proved similarly in the split Cartan case (since $\operatorname{Im}\rho_{E,p^2} \not\subseteq C_{sp}^+(49)$ by \Cref{thm: rszb}).
\end{proof}

To prove \Cref{thm: Cns+(49)}, we first show that each rational point on $X_{ns}^+(49)$ gives rise to a primitive integral solution of the generalised Fermat equation $a^2 + 28b^3 = 27c^7$. We then solve this equation by adapting the strategy employed by Poonen--Schaefer--Stoll \cite{PSS07} to find the primitive solutions of $a^2+b^3=c^7$.
The idea is to show that each solution to this equation gives rise, via a construction involving modularity of elliptic curves over $\Q$, to a rational point on one of finitely many twists of the Klein quartic (i.e.,~the plane curve $x^3y + y^3z + z^3x = 0$). Via this construction, solving the generalised Fermat equation $a^2+28b^3=27c^7$ reduces to the determination of the rational points on a finite number of plane quartics, which in our case can be handled using the Chabauty-Coleman approach.

We also prove that a similar argument applies to the modular curves $X_{ns}^\#(49)$ and $X_{sp}^\#(49)$: the rational points on these two curves give rise to primitive integral solutions of one of the equations $a^2 + 196b^3 = 27 c^7$ and $a^2 + 28b^3 = 27c^7$ (the latter only in the case of $X_{sp}^\#(49)$). By applying again the strategy outlined above, we reduce the resolution of these equations to the determination of the rational points on one further twist of $X(7)$. Unfortunately, we cannot determine the rational points of this twist, but we conjecture that there are exactly $4$.

\begin{conjecture}\label{conj: XE3}
    The set of rational points of the curve
    $$C : \ \ x^4 + 3x^3y - 3x^2yz - 3x^2z^2 + 6xy^3 - 6xy^2z + 3xyz^2 - 2xz^3 + 4y^4 + 2y^3z - 5yz^3 = 0$$
    is $C(\Q) = \{ [0 : 0 : 1], [1 : 1 : 1], [2 : 0 : 1],  [-1 : 0 : 1] \}$.
\end{conjecture}

A proof of \Cref{conj: XE3} would lead to a complete classification of the images of $7$-adic Galois representations attached to elliptic curves over $\Q$, as the following result shows.

\begin{theorem}\label{thm: conjectural 7-adic images}
    Let $p$ be a prime, let $E/\Q$ be an elliptic curve without complex multiplication, and let $G:=\operatorname{Im}\rho_{E,p^\infty}$. 
    Assume \Cref{conj: XE3}.
    Exactly one of the following holds:
    \begin{enumerate}[(i)]
        \item the modular curve $X_G$ is isomorphic to $\mathbb{P}^1$ or an elliptic curve of rank one, in which case $G$ is identified in \cite[Corollary 1.6]{sutherlandzywina17} and $\langle G, -I \rangle$ is listed in \cite[Tables 1--4]{sutherlandzywina17};
        \item the modular curve $X_G$ has finitely many non-CM rational points, and $G$ is listed in \cite[Table 1]{rszb22} with the non-CM rational points of $X_G$ (expressed as $j$-invariants or models of the corresponding elliptic curves);
        \item $G \pmod {p^n}$ is contained in $C_{ns}^+(p^n)$ for some $p^n \in \{5^2, 11^2\} \cup \{p \text{ prime} \mid p \ge 19\}$, with $C_{ns}^+(p^n)$ as in Definitions \ref{def: p-adic groups} and \ref{def:cartan}.
    \end{enumerate}
\end{theorem}

\medskip

We now quickly explain the strategy behind the proofs of \Cref{thm: Cns+(49),thm: conjectural 7-adic images}.

\begin{enumerate}[(1)] 

    \item As proved in \cite[Proposition 5.13]{furio24}, for every elliptic curve $E/\Q$ such that $\operatorname{Im}\rho_{E,p^n} \subseteq C_{ns}^+(p^n)$ for some odd prime power $p^n$, the $j$-invariant $j(E)$ is a rational number whose denominator is a $p^n$-th power, i.e.,~$j(E) = \frac{a}{b^{p^n}}$ for relatively prime integers $a,b$. Similarly, in \Cref{prop: ciuffetti denominators} we prove that if $\operatorname{Im}\rho_{E,7^2} \subseteq G$, where $G$ is the subgroup corresponding to one of the modular curves $X_{ns}^\#(49)$ or $X_{sp}^\#(49)$, the denominator of $j(E)$ is a $49$-th power.

    \item The modular curve $X_{ns}^+(49)$ admits a map to the curve $X_{ns}^+(7)$, which is a genus-0 modular curve with a $j$-map to $X(1)$ expressed by the rational function $j_{ns}(t) = \frac{g(t)^3}{f(t)^7}$ (see \Cref{thm: j-invariants Zywina}). This means that for every elliptic curve $E/\Q$ such that $\operatorname{Im}\rho_{E,7^2} \subseteq C_{ns}^+(7^2)$ there exists $t \in \mathbb{P}^1(\Q)$ such that $j(E) = \frac{g(t)^3}{f(t)^7}$ is a rational number whose denominator is a $49$-th power. A similar conclusion holds for the modular curves $X_{ns}^\#(49)$ and $X_{sp}^\#(49)$.

    \item Since $g(t)^3$ and $f(t)^7$ have the same degree, we can write $t=\frac{x}{y}$, with $x,y$ coprime integers, and express $j(E)$ as the ratio $\frac{G(x,y)^3}{F(x,y)^7}$, where $F$ and $G$ are two-variable polynomials with integral coefficients, given respectively by the homogenisations of $f$ and $g$. Computing the resultant of $f$ and $g$, we find that the gcd of $F(x,y)$ and $G(x,y)$ is $k \in \{1,8\}$. Since the denominator of $j(E)$ is a $49$-th power, this implies that there exists an integer $z$ such that $F(x,y) = kz^7$. This argument is explained in \Cref{prop: norm equations}.

    \item The equation $F(x,y) = kz^7$ has finitely many solutions, hence it suffices to find them to determine the rational points of $X_{ns}^+(49)$. Some elementary arithmetic arguments allow us to show that every primitive solution to this equation corresponds to a primitive solution of the generalised Fermat equation $a^2 + 28b^3 = 27 c^7$. Applying the same argument to the modular curves $X_{ns}^\#(49)$ and $X_{sp}^\#(49)$ we obtain the Fermat equation $a^2 + 196b^3 = 27 c^7$. This reduction is explained in \Cref{prop: from norm to Fermat} and \Cref{cor: specific norm equations}.

    \item We can now apply the strategy of Poonen, Schaefer, and Stoll to the equation $a^2 + 28b^3 = 27c^7$. This strategy can be seen as a generalisation of the modular proof of Fermat's last theorem. To each solution of the equation $a^2+28b^3=27c^7$ (or its twisted version $a^2 + 196b^3 = 27c^7$) we attach an elliptic curve $E_{(a,b,c)}$ over $\Q$ whose $j$-invariant encodes the solution $(a,b,c)$. Due to the special form of $E_{(a,b,c)}$, level-lowering results imply that the mod-7 representation of $E_{(a,b,c)}$ must arise from one of finitely many modular forms. For the cases of interest, one can show that the modular form giving rise to $E_{(a,b,c)}[7]$ can be taken to be a weight-2 newform $f$ with rational coefficients, so that (by modularity) there is some other elliptic curve $E_f/\Q$ with $E_{(a,b,c)}[7] \cong E_f[7]$. The crucial point here is that there is a finite, known list of forms $f$, and therefore a finite, known list of possible curves $E_f$.

    \item Finally, elliptic curves over $\Q$ whose mod-7 Galois representations are isomorphic to $E_f[7]$ correspond to rational points on suitable twists of the modular curve $X(7)$, which is isomorphic to the Klein quartic (i.e., the plane curve $x^3y + y^3z + z^3x = 0$). Putting everything together, the solution of the generalised Fermat equation of signature $(2, 3, 7)$ can be reduced to the determination of the rational points on finitely many twists of the Klein quartic. Both in \cite{PSS07} and in our case $a^2 + 28b^3 = 27c^7$, this is a task that can be completed by a combination of 2-descent, Mordell-Weil sieve, and Coleman's effective version of Chabauty's method. In the case $a^2 + 196b^3 = 27c^7$, we obtain an additional twist (namely, the curve of \Cref{conj: XE3}) of which we are not able to determine the rational points 
\end{enumerate}

The paper is structured as follows. We begin in \Cref{sect: preliminaries} by reviewing known results on the $p$-adic representations of elliptic curves over $\Q$ and slightly extending them to include the Galois representations associated with the modular curves $X_{ns}^\#(49)$ and $X_{sp}^\#(49)$ (corresponding to step (1) above). In \Cref{sect: ratpts to Fermat}, we carry out steps (2)--(4), showing that each rational point on the modular curves $X_{ns}^+(49)$, $X_{ns}^\#(49)$, and $X_{sp}^\#(49)$ yields a primitive solution to a certain generalised Fermat equation. In \Cref{sect: Fermat to ell curves} we implement step (5), associating with every solution of our Fermat equations suitable elliptic curves over $\Q$. Any such elliptic curve arises from a modular form, and we use several arguments (inertia action, symplectic criteria, global computation of torsion fields...) to reduce the list of relevant modular forms to just three. In \Cref{sect: ell curves to ratpts genus 3} we recall how specific twists of $X(7)$ parametrise elliptic curves over $\Q$ with a given mod-7 Galois representation and write down the finitely many twists that we need to consider (see step (6) above). Finally, in \Cref{sect: genus 3} we complete the proof of \Cref{thm: Cns+(49)} by determining the rational points on all but one of these twists.

\subsection{Possible extensions}

One may wonder whether the methods of this paper apply to other modular curves $X_{ns}^+(p^n)$ with $p \ne 7$ and $n \geq 2$ (by \Cref{thm: rszb}(iii), the most interesting cases are $p^n \in \{3^3, 5^2, 11^2\}$). Unfortunately, it seems that their application is not straightforward for different reasons.

The curve $X_{ns}^+(27)$ admits a map to the modular curve $X_{ns}^+(9)$, which is again isomorphic to $\mathbb{P}^1$. In suitable coordinates, the $j$-map $X_{ns}^+(9) \to \mathbb{P}^1$ is of the form $\frac{G(x,y)}{F(x,y)^9}$ (see for example \cite{baran09}), where $F(x,y)$ is a homogeneous polynomial of degree $3$ (indeed, the degree of $F$ equals the number of cusps of the curve $X_{ns}^+(9)$, which is $3 = [\Q(\zeta_9)^+ : \Q]$). We then obtain two equations $F(x,y) = kz^3$, for $k \in \{1,3\}$. These equations define two elliptic curves: the first one, for $k=1$, has rank $1$, while the second one, for $k=3$, has rank $0$. This implies that there are infinitely many rational points on $X_{ns}^+(9)$ whose denominator is a $27$-th power. We also note that the rational points on $X_{ns}^+(27)$ have recently been determined via the quadratic Chabauty method \cite{balakrishnan25}.

In the case of the curve $X_{ns}^+(25)$, there is again a map to the genus-0 curve $X_{ns}^+(5)$. The $j$-map from this curve to $\mathbb{P}^1$ is of the form $\frac{G(x,y)}{F(x,y)^5}$ (see for example \cite[Theorem 1.4]{zywina15}), where $F$ is a homogeneous polynomial of degree $2$ (which is the number of cusps of $X_{ns}^+(5)$). Imposing that the denominator of the $j$-invariant is a $25$-th power then leads to an equation of the form $f(x,y)=k z^5$, where $f(x,y)$ is now a binary \textit{quadratic} form, and equations of this form can have infinitely many solutions.

Finally, in the case of $X_{ns}^+(121)$, we have a map to $X_{ns}^+(11)$, which is an elliptic curve, making it harder to analyse the $j$-map to $\mathbb{P}^1$.

In a slightly different direction, it should be possible to adapt these ideas to study Cartan structures of level $7p$ for $p > 7$: by \Cref{prop: cartan49 denominator}, if $E/\Q$ has a non-split Cartan structure both at $7$ and at $p$, then the denominator of the $j$-invariant is a $(7p)$-th power. This leads to equations of the form $F(x,y) = k z^p$, where $F(x,y)$ is a (fixed) homogeneous form of degree 3 and $k$ belongs to a small, explicit list of integers. The main result of \cite{MR2999040} shows that in many cases such equations are only solvable for bounded values of $p$; unfortunately, this does not apply here, because $F(1,0)=1$, so $(x,y,z,p)=(1,0,1,p)$ is a solution for all $p$, at least for $k=1$. On the other hand, it is not hard to show that the $abc$ conjecture implies that $F(x,y) = k z^p$ only has trivial solutions for sufficiently large $p$, which leads to the conclusion that (under the $abc$ conjecture) $X_{ns}^+(7p)(\Q)$ consists of CM points for sufficiently large $p$.

\begin{proposition}\label{prop: cartan7p}
    Assume the $abc$ conjecture. For all sufficiently large primes $p$, the only rational points on $X_{ns}^+(7) \times_{X(1)} X_{ns}^+(p)$ are CM points.
\end{proposition}

We will give the proof of this fact in the appendix.
Although conditional on $abc$, this proposition is interesting in that it describes the rational points on an \textit{infinite family} of non-split Cartan modular curves, many of which do have rational (CM) points.
No other result of this flavour seems to be known in the non-split Cartan case.
The method of proof of \Cref{prop: cartan7p} also yields information for fixed values of $p$, and it would be worthwhile to examine specific cases in detail. For example, when $p=5$, the determination of the rational points on $X_{ns}^+(7) \times_{X(1)} X_{ns}^+(5)$ reduces to solving equations of the form $F(x,y) = kz^5$, where $F(x,y)$ is a homogeneous polynomial of degree 3. We point out that in the case $p=3$, the resulting equations define (cones over) genus 1 curves, one of which is an elliptic curve with positive rank. One can check that this elliptic curve is in fact isomorphic to $X_{ns}^+(7) \times_{X(1)} X_{ns}^+(3)$.

By an argument similar to that of \Cref{prop: cartan7p} one can prove, assuming again the $abc$ conjecture, that for sufficiently large $p$ the only rational points of the curve $X_{sp}^+(7) \times_{X(1)} X_{ns}^+(p)$ are CM points. However, a more general and unconditional result was proved by Lemos in \cite{lemos19split}.
Nevertheless, this strategy could be applied to determine rational points on the curves $X \times_{X(1)} X_{ns}^+(p)$, where $X$ is a genus-0 modular curve with $j$-map of the form $\frac{G(x,y)}{F(x,y)}$, where $F$ is a homogeneous polynomial with at least $3$ distinct roots (see for example \cite[Theorem 1]{darmon-granville95}, which guarantees that this is a sufficient condition to ensure finiteness of solutions of the equations $F(x,y) = kz^p$ for $p$ sufficiently large).

\paragraph{Computer-assisted calculations.} Several of the arguments in this paper depend on computer calculations, performed in MAGMA \cite{MR1484478}. All the scripts to verify our claims are available at
\begin{center}
    \url{https://github.com/DavideLombardoMath/7-adic-representations/}.
\end{center}

\paragraph{Acknowledgements.} We thank Mart\'in Azon, Jennifer Balakrishnan, Niels Bruin, Matteo Longo, and Jan Steffen M\"uller for interesting conversations related to the topic of this paper. We are grateful to Abbey Bourdon, Özlem Ejder, Elisa Lorenzo García, \'Elie Studnia, and Andrew Sutherland for their helpful comments on an earlier version of this manuscript and for drawing our attention to relevant references.
The first author is supported by the grant ANR-HoLoDiRibey of the Agence Nationale de la Recherche (PI Gregorio Baldi).
The second author is supported by MUR grant PRIN-2022HPSNCR (funded by the European Union project Next Generation EU) and is a member of the INdAM group GNSAGA.

\section{Preliminaries}\label{sect: preliminaries}

We collect in this section some known results about the mod-$p$ (and especially mod-$7$) images of Galois of elliptic curves over $\Q$.
The following result is part of \cite[Theorem 1.5]{zywina15}.

\begin{theorem}[Zywina]\label{thm: j-invariants Zywina}
    Let $E/\Q$ be an elliptic curve.
    \begin{itemize}
        \item $\operatorname{Im}\rho_{E,7}$ is contained in a Borel subgroup if and only if there exists $t \in \Q^\times$ such that $$j(E) = j_B(t) = \frac{(t^2+245t+2401)^3(t^2+13t+49)}{t^7}.$$
        \item $\operatorname{Im}\rho_{E,7}$ is contained in the normaliser of a split Cartan subgroup if and only if there exists $t \in \Q$ such that $$j(E) = j_{sp}(t) = \frac{t(t+1)^3(t^2-5t+1)^3(t^2-5t+8)^3(t^4-5t^3+8t^2-7t+7)^3}{(t^3-4t^2+3t+1)^7}.$$
        \item $\operatorname{Im}\rho_{E,7}$ is contained in the normaliser of a non-split Cartan subgroup if and only if there exists $t \in \mathbb{P}^1(\Q)$ such that $$j(E) = j_{ns}(t) = \frac{64t^3(t^2+7)^3(t^2-7t+14)^3(5t^2-14t-7)^3}{(t^3-7t^2+7t+7)^7}.$$
    \end{itemize}
\end{theorem}

Zywina's theorem is stated for non-CM elliptic curves. However, one can easily verify that CM $j$-invariants also satisfy the conclusion of the statement: indeed, the maps given above are simply the $j$-maps from the modular curves $X_0(7)$, $X_{sp}^+(7)$, $X_{ns}^+(7)$ (which are all isomorphic to $\mathbb{P}^1$, with coordinate $t$) to $X(1)$.

\begin{remark}\label{rmk: extreme values}
    In the case of the function $j_{ns}(t)$ we easily compute
    \[
    j_{ns}(0)=0, \quad j_{ns}(\infty)=8000,
    \]
    which are both CM $j$-invariants. For $j_{sp}$ we have instead $j_{sp}(0)=0$ and $j_{sp}(\infty)=\infty$.
\end{remark}

\begin{lemma}\label{lemma: 7 isogenies}
    Let $E/\Q$ be an elliptic curve. If $E$ admits a rational $7$-isogeny, then $v_2(j(E)) \le 0$.
\end{lemma}
\begin{proof}
    By Theorem \ref{thm: j-invariants Zywina} there exists $t \in \Q^\times$ such that $$j(E) = j_B(t) = \frac{(t^2+245t+2401)^3(t^2+13t+49)}{t^7} = \frac{f(t)}{t^7}.$$
    If $v_2(t) < 0$, then $v_2(f(t)) = 8v_2(t)$, and so $v_2(j(t))=8v_2(t)-7v_2(t) = v_2(t) < 0$. If instead $v_2(t) \ge 0$, we have $v_2(f(t)) = 0$ and therefore $v_2(j(t)) = -7v_2(t) \le 0$.
\end{proof}

\begin{proposition}\label{prop: j and Hasse}
    Let $p>3$ be a prime, $K$ a $p$-adic field with uniformiser $\pi_K$, and $E/K$ an elliptic curve with potentially good reduction.
    \begin{itemize}
        \item Suppose that $j(E) \equiv 0 \pmod {\pi_K}$: then $E$ has potentially supersingular reduction if and only if $p \equiv 2 \pmod 3$.
        \item Suppose that $j(E) \equiv 1728 \pmod {\pi_K}$: then $E$ has potentially supersingular reduction if and only if $p \equiv 3 \pmod 4$.
    \end{itemize}
\end{proposition}

\begin{proof}
    Let $L/K$ be a finite extension over which $E$ acquires good reduction and let $\tilde{E}$ be the special fibre of a model of $E_L$ having good reduction. Writing $\pi_L$ for a uniformiser of $L$, we have $j(\tilde{E}) = j(E) \bmod (\pi_L)$, so in the two cases $j(\tilde{E})$ is $0$ or $1728$. In particular, $\tilde{E}$ is a twist of $y^2=x^3+1$ (respectively $y^2=x^3+x$): it is well-known that this curve is superingular over a finite field of characteristic $p$ precisely when $p \equiv 2 \pmod{3}$ (resp.~$p\equiv 3 \pmod 4$).
\end{proof}

\subsection{$7$-adic images of Galois representations}

Following \cite{furio24}, which builds on the work of Zywina \cite{zywina11}, we recall a classification of subgroups of $\GL_2(\Z_p)$ that includes all groups that can arise as $p$-adic Galois images of non-CM elliptic curves over~$\Q$.

\begin{definition}\label{def: p-adic groups}
	Given a prime $p$ and a subgroup $G < \GL_2(\Z_p)$, for every $n \ge 1$ we define:
	\begin{itemize}
		\item $G(p^n):= G \pmod {p^n} \subseteq \GL_2\left({\Z}/{p^n\Z}\right)$;
		\item $G_n := \left\lbrace A \in G \mid A \equiv I \pmod {p^n} \right\rbrace$, where $I$ is the identity matrix.
	\end{itemize}
    Note that $G(p^n)=G/G_n$.
\end{definition}

\begin{definition}\label{def:cartan}
	Let $p$ be an odd prime and let $\varepsilon$ be the reduction modulo $p$ of the least positive integer which represents a quadratic non-residue in $\F_p^\times$. We define the following subgroups of $\GL_2(\Z_p)$:
	\begin{align*}
		\text{split Cartan:} \qquad &C_{sp}:= \left\lbrace \begin{pmatrix} a & 0\\ 0 & b \end{pmatrix} \,\middle|\, a,b \in \Z_p^\times \right\rbrace, \\
		\text{non-split Cartan:}  \qquad &C_{ns}:= \left\lbrace \begin{pmatrix} a & \varepsilon b \\ b & a \end{pmatrix} \,\middle|\, a,b \in \Z_p, \ (a,b) \not\equiv (0,0) \mod p \right\rbrace.
	\end{align*}
	Further set $C_{sp}^+ := C_{sp} \cup \begin{pmatrix} 0 & 1 \\ 1 & 0 \end{pmatrix} C_{sp}$ and $C_{ns}^+:= C_{ns} \cup \begin{pmatrix} 1 & 0 \\ 0 & -1 \end{pmatrix} C_{ns}$. These are the normalisers in $\GL_2(\Z_p)$ of $C_{sp}$ and $C_{ns}$, respectively.
\end{definition}

Throughout the rest of the paper we will make the following identifications. 
\begin{itemize}
    \item Given a class $a \in \Z/p\Z$, we denote by $pa \in \Z/p^2\Z$ the class modulo $p^2$ of $p\tilde{a}$, where $\tilde{a} \in \Z$ is any lift of $a \in \Z/p\Z$. 
    \item
    Let $\GL_2(\Z/p^2\Z)_1 :=\operatorname{ker}\left( \operatorname{GL}_2(\Z/p^2\Z) \to \operatorname{GL}_2(\Z/p\Z)\right)$. There is a group isomorphism 
    \begin{align*}
        \begin{array}{ccc}
            \left(\GL_2(\Z/p^2\Z)_1 , \times \right) & \longrightarrow & \left( M_2(\Z/p\Z), + \right) \\
            \begin{pmatrix}
                1 + pa_{11} & pa_{12} \\
                pa_{21} & 1+pa_{22}
            \end{pmatrix} & \longleftrightarrow & \begin{pmatrix}
                a_{11} & a_{12} \\
                a_{21} & a_{22}
            \end{pmatrix} \\
            I + pA & \longleftrightarrow & A
        \end{array}
    \end{align*}
    Since this is a group isomorphism, we can describe subgroups of $\GL_2(\Z/p^2\Z)_1$ as $\{I + pA \mid A \in V \}$, where $V$ is a vector subspace of $M_2(\F_p)$. We will denote such a subgroup simply by $I + pV$.
\end{itemize}

\begin{definition}\label{def: ciuffetto}
    Let $p$ and $\varepsilon$ be as in Definition \ref{def:cartan}. Let $V_{sp}$ and $V_{ns}$ be the additive groups of matrices with coefficients in $\F_p$ of the form
    \begin{align*}
		\text{split Cartan:} \qquad &V_{sp}:= \left\lbrace \begin{pmatrix} a & b\\ c & a \end{pmatrix} \,\middle|\, a,b,c \in \F_p \right\rbrace, \\
		\text{non-split Cartan:} \qquad &V_{ns}:= \left\lbrace \begin{pmatrix} a & \varepsilon b \\ -b & d \end{pmatrix} \,\middle|\, a,b,d \in \F_p \right\rbrace.
	\end{align*}
    We define the group $G_{ns}^\#(p^2)$ as the unique subgroup of $\GL_2(\Z/p^2\Z)$ of order $2p^3(p^2-1)$ containing $I + pV_{ns}$ and whose projection modulo $p$ equals $C_{ns}^+(p)$. In particular, we have $G_{ns}^\#(p^2) \cong C_{ns}^+(p) \ltimes V_{ns}$, where the semidirect product is defined by the conjugation action (see for example \cite[Proposition 2.2]{furio24}). 
    We define $G_{sp}^\#(p^2)$ in a similar way, as the unique subgroup of $\GL_2(\Z/p^2\Z)$ containing $I+pV_{sp}$ whose projection modulo $p$ equals $C_{sp}^+(p)$.
\end{definition}

\begin{remark}
The groups $G_{sp}^\#(7^2)$ and $G_{ns}^\#(7^2)$ correspond respectively to the modular curves \verb|49.196.9.1| and \verb|49.147.9.1|. 
    This motivates our notation $X_{sp}^\#(49)$ and $X_{ns}^\#(49)$ for these modular curves, which correspond to case (iv) of \Cref{thm: rszb}.
\end{remark}

The following is \cite[Proposition 5.13]{furio24}.

\begin{proposition}\label{prop: cartan49 denominator}
    Let $E/\Q$ be an elliptic curve and let $p$ be an odd prime such that $\pnRep{n} \subseteq C_{ns}^+(p^n)$ for some positive integer $n$. Write $j(E) = \frac{a}{b}$ with $\gcd(a,b)=1$: then $b=c^{p^n}$ for some integer $c$.
\end{proposition}

We now generalise this result to the groups $G_{ns}^\#(p^2)$ and $G_{sp}^\#(p^2)$.

\begin{proposition}\label{prop: ciuffetti denominators}
    Let $E$ be an elliptic curve defined over a number field $K$ and let $p$ be an odd prime such that $\pnRep{2}$ is contained in $G_{ns}^\#(p^2)$ or in $G_{sp}^\#(p^2)$ up to conjugacy. In the split case, assume also that $p$ is tamely ramified in $K$. Then $p^2$ divides
    $$\gcd_{\lambda \subseteq \OK \text{ prime}}(\max\{0,-v_\lambda(j(E))\}).$$
\end{proposition}

\begin{proof}
    If $j(E) \in \OK$ the statement is trivial, so suppose otherwise. Let $\lambda$ be a prime of $K$ such that $e := -v_\lambda(j(E)) > 0$. We need to show that $p^2 \mid e$. We can assume that $\zeta_p \in K$: indeed, $p$ does not divide $[K(\zeta_p):K]$, and so the power of $p$ that divides the valuation of $j(E)$ at primes of $K(\zeta_p)$ above $\lambda$ is the same as the power of $p$ dividing the $\lambda$-adic valuation of $j(E)$.
    Consider $E$ to be defined over $K_\lambda$ and let $E_q$ be the Tate curve with parameter $q \in K_\lambda^\times$, isomorphic to $E$ over an at most quadratic extension of $K_\lambda$ (see \cite[Chapter V, \S3]{silverman-advanced-topics} for details). We know that $v_\lambda(q) = e$. If $\chi_{p^2}$ is the cyclotomic character modulo $p^2$, there is an at most quadratic character $\psi$ such that $\rho_{E_q,p^2} \cong \rho_{E,p^2} \otimes \psi$, hence we have
	$$\rho_{E,p^2} \otimes \psi \cong \begin{pmatrix} \chi_{p^2} & k \\ 0 & 1 \end{pmatrix} \equiv \begin{pmatrix} 1 & k \\ 0 & 1 \end{pmatrix} \pmod p,$$
	where $k(\sigma)$ is defined by the equality $\sigma\left(q^\frac{1}{p^2}\right) = q^\frac{1}{p^2} \zeta_{p^2}^{k(\sigma)}$. Note that the character $\chi_p = \chi_{p^2} \pmod p$ is trivial, as $\zeta_p \in K$. Moreover, $k \equiv 0 \pmod p$, for otherwise $\pRep$ would contain elements of order $p$, which is impossible since $\pRep \subseteq C_{ns}^+(p) \cup C_{sp}^+(p)$.
    In particular, every element in $\operatorname{Im}\rho_{E_q,p^2}$ can be written as $I+pA$, with $A$ of the form $\begin{pmatrix} \ast & \ast \\ 0 & 0 \end{pmatrix}$. Under the isomorphism $I+pA \mapsto A \pmod p$, we deduce that $\operatorname{Im}\rho_{E_q,p^2}$ can be viewed as a subgroup of the additive group $\F_p\begin{pmatrix} 1 & 0 \\ 0 & 0 \end{pmatrix} \oplus \F_p \begin{pmatrix} 0 & 1 \\ 0 & 0 \end{pmatrix} \subseteq M_2(\F_p)$, which has order $p^2$. The group $\operatorname{Im}\rho_{E_q,p^2}$ must be conjugate to a subgroup of either $V = V_{sp}$ or $V = V_{ns}$ (see \Cref{def: ciuffetto}).
    Suppose that $\operatorname{Im}\rho_{E_q,p^2}$ has order $p^2$. This implies that there exists a subgroup $H$ of $V$ of order $p^2$ every element of which has rank $0$ or $1$. Let $M_1, M_2$ be an $\F_p$-basis of $H$ and write $M_i=\begin{pmatrix} a_i & b_i \\ c_i & d_i\end{pmatrix}$ for $i=1, 2$. Changing basis if necessary (recall that $\operatorname{Im}\rho_{E_q,p^2}$ has a natural structure of $\F_p$-vector space) we can assume that $\operatorname{tr}M_1 = 0$, and so that $d_1 = -a_1$. For every choice of $x,y \in \F_p$ we must have
    $$0 = \det(xM_1+yM_2) = (a_1x+a_2y)(d_1x+d_2y)-(b_1x+b_2y)(c_1x+c_2y) = (a_1d_2 + a_2d_1 - b_1c_2 - b_2c_1)xy = 0,$$
    where we used the fact that $\det M_i = a_id_i-b_ic_i = 0$. In particular, this gives $a_1d_2+a_2d_1 = b_1c_2+b_2c_1$. \\
    In the non-split case, where $b_i = -\varepsilon c_i$, using the fact that $d_1 = -a_1$, we would have $0 = \det M_1 = -a_1^2 + \varepsilon c_1^2$, which is impossible by definition of $\varepsilon$ (since $a_1$ and $c_1$ cannot be both equal to $0$, otherwise $M_1 = 0$). \\
    In the split case, where $a_i = d_i$, we have $a_1 = -d_1 = -a_1$, and so $a_1 = 0$. This gives $0 = \det M_1 = -b_1c_1$, and so exactly one of $b_1$ and $c_1$ is $0$. Suppose without loss of generality that $c_1=0$. The equation $a_1d_2+a_2d_1 = b_1c_2+b_2c_1$ becomes $0 = b_1c_2$, and so $c_2=0$. This implies that $0 = \det M_2 = a_2^2$, and so
    $$M_2 = \begin{pmatrix} 0 & b_2 \\ 0 & 0 \end{pmatrix} \in \F_p \begin{pmatrix} 0 & b_1 \\ 0 & 0 \end{pmatrix} = \F_p M_1,$$
    which is impossible because $M_1$ and $M_2$ are a basis (if we had assumed that $b_1 = 0$ we would have obtained the same result). \\
    We have then proved that $\operatorname{Im}\rho_{E_q,p^2}$ has order either $1$ or $p$. In the first case, we have $k=0$, and so $\sigma(q^{\frac{1}{p^2}}) = q^{\frac{1}{p^2}}$ for every automorphism $\sigma$. This implies that $q^\frac{1}{p^2} \in K_\lambda$, which gives $v_\lambda(q^{\frac{1}{p^2}}) = \frac{e}{p^2} \in \Z$, and hence $p^2 \mid e$. In the second case, let
	\begin{equation*}
		M_\sigma := (\rho_{E,p^2} \otimes \psi)(\sigma) = I + p \begin{pmatrix} a & b \\ 0 & 0 \end{pmatrix}
	\end{equation*}
	be a generator of the group. The matrix $M_\sigma$ is the reduction modulo $p^2$ of $(\rho_{E,p^\infty} \otimes \psi)(\sigma)$, so (by a slight abuse of notation) we can also consider $a,b$ as elements of $\Z_p$.
    We can assume that $a \not \equiv 0 \pmod{p}$: indeed, in the split case there exists an automorphism $\sigma$ such that $\chi_{p^2}(\sigma) = 1+p$ (because $p$ is tamely ramified in $K$), while in the non-split case, $a \equiv 0 \pmod{p}$ would give a non-zero element $g$ in $V$ conjugate to $\begin{pmatrix} 0 & 1 \\ 0 & 0 \end{pmatrix}$, and so with $\operatorname{tr} g = \det g = 0$, which is impossible as shown above.
	Note that $a$ is invertible in $\Z_p$ and consider the element $q^\frac{1}{p^2}\zeta_{p^2}^{-{b}/{a}}$: in the basis $( \zeta_{p^2}, q^\frac{1}{p^2} )$ this is expressed as the vector $\left(-{b}/{a} , 1\right)$. 
    We then have
	\begin{align*}
		\sigma(q^\frac{1}{p^2}\zeta_{p^2}^{-{b}/{a}}) \longleftrightarrow \begin{pmatrix} 1+p a & p b \\ 0 & 1 \end{pmatrix} \begin{pmatrix} -{b}/{a} \\ 1 \end{pmatrix} = \begin{pmatrix} -{b}/{a} \\ 1 \end{pmatrix} \longleftrightarrow q^\frac{1}{p^2}\zeta_{p^2}^{-{b}/{a}},
	\end{align*}
	and so $\sigma(q^\frac{1}{p^2}\zeta_{p^2}^{-{b}/{a}}) = q^\frac{1}{p^2}\zeta_{p^2}^{-{b}/{a}}$. Since $M_\sigma$ generates $\operatorname{Im}\rho_{E_q,p^2}$, this implies that $q^\frac{1}{p^2}\zeta_{p^2}^{-{b}/{a}}$ is fixed by every automorphism, and so $q^\frac{1}{p^2}\zeta_{p^2}^{-{b}/{a}} \in K_\lambda$. We conclude that $v_\lambda(q^\frac{1}{p^2}\zeta_{p^2}^{-{b}/{a}}) = \frac{e}{p^2} \in \Z$, and hence $p^2 \mid e$.
\end{proof}

\begin{lemma}\label{lemma: 7-adic valuation on X_ns^++(49)}
    Let $E/\Q$ be an elliptic curve with $\operatorname{Im}\rho_{E,49} \subseteq G_{ns}^\#(49)$. We have $v_7(j(E)) > 0$.
\end{lemma}

\begin{proof}
    Set $X:=X_{ns}^\#(49)$. We consider the canonical model of the curve $X$ and its $j$-map to $\mathbb{P}^1$ (which can be found at \cite[\href{https://beta.lmfdb.org/ModularCurve/Q/49.147.9.a.1/}{modular curve 49.147.9.a.1}]{lmfdb}), and write $j$ as the ratio $\frac{j_0}{j_1}$ of two homogeneous polynomials $j_0, j_1$ with coefficients in $\Z$. We na\"ively extend $X$ over $\Z$ by considering the projective scheme $\mathcal{X} / \operatorname{Spec} \mathbb{Z}$ defined by the same equations defining the canonical model of $X$ (rescaled to have integer coefficients). 
    It then suffices to check that for every $P \in \mathcal{X}(\Z/49\Z)$ we have $j_0(P) \equiv 0 \pmod 7$ and $j_1(P) \not\equiv 0 \pmod 7$.
\end{proof}

\section{From rational points on modular curves to Fermat equations}\label{sect: ratpts to Fermat}

\begin{proposition}\label{prop: norm equations}
    Let $E/\Q$ be an elliptic curve without CM and let $j_{sp}$ and $j_{ns}$ be defined as in \Cref{thm: j-invariants Zywina}.
    \begin{itemize}
        \item If $\operatorname{Im}\rho_{E,49} \subseteq C_{ns}^+(49)$, there exist $k \in \{1,8\}$ and pairwise coprime integers $x,y,z \in \Z$ such that
        $$x^3-7x^2y+7xy^2+7y^3 = k \cdot z^7 \qquad \text{and} \qquad j(E) = j_{ns}\left(\frac{x}{y}\right).$$
        \item If $\operatorname{Im}\rho_{E,49} \subseteq G_{ns}^\#(49)$, there exist $k \in \{1,8\}$ and pairwise coprime integers $x,y,z \in \Z$ such that
        $$x^3-7x^2y+7xy^2+7y^3 = 7k \cdot z^7 \qquad \text{and} \qquad j(E) = j_{ns}\left(\frac{x}{y}\right).$$
        \item If $\operatorname{Im}\rho_{E,49} \subseteq G_{sp}^\#(49)$, there exist $k \in \{1,7\}$ and pairwise coprime integers $x,y,z,w \in \Z$ such that
        $$x^3-4x^2y+3xy^2+y^3 = k \cdot z^7, \qquad y=w^7, \qquad \text{and} \qquad j(E) = j_{sp}\left(\frac{x}{y}\right).$$
    \end{itemize}
\end{proposition}

\begin{proof}
    Let $X =X_{ns}^+(49)$ or $X=X_{ns}^\#(49)$. This curve admits a map to $X_{ns}^+(7)$. If $E/\Q$ is an elliptic curve with $\operatorname{Im}\rho_{E,49}$ contained in either $C_{ns}^+(49)$ or $G_{ns}^\#(49)$, this corresponds to a rational point on the modular curve $X$, and in turn to a rational point on the modular curve $X_{ns}^+(7)$.
    In particular, by \Cref{thm: j-invariants Zywina} there exists a number $t = \frac{x}{y} \in \mathbb{P}^1(\Q)$ such that
    \[
    j(E) = j_{ns}(t) =  \frac{g(x,y)^3}{f(x,y)^7},
    \]
    where 
    \[
    g(x,y) = 4x(x^2+7y^2)(x^2-7xy+14y^2)(5x^2-14xy-7y^2), \quad\quad f(x,y) = x^3-7x^2y+7xy^2+7y^3,
    \]
    and $x, y$ are coprime integers. Since $E$ is not CM, by \Cref{rmk: extreme values} we have $xy\neq 0$.
    The resultant of $f$ and $g$, considered as polynomials in $x$ with coefficients in the ring $\Z[y]$, is $-2^{21} \cdot 7^7 \cdot y^{10}$. It is clear that $f(x,y)$ and $y$ are coprime, hence the greatest common divisor of $f(x,y)$ and $g(x,y)$ must divide $2^{21} \cdot 7^7$. It is easy to see that $f(x,y) \equiv 0 \pmod{49}$ implies $x \equiv y \equiv 0 \pmod{7}$, while we are assuming $(x,y)=1$.
    Similarly, one can check that $f(x,y) \equiv 0 \pmod{2^4}$ implies $x \equiv y \equiv 0 \pmod{2}$, which is again impossible.
    By \Cref{prop: cartan49 denominator,prop: ciuffetti denominators} we know that the denominator of $j(E)$ is a $49$-th power. There is no simplification in the ratio $g(x,y)^3 / f(x,y)^7$ except possibly at the primes $2$ and $7$, so for any prime $p \not \in \{2,7\}$ that divides $f(x,y)$ we must have $v_p(f(x,y)^7) \equiv 0 \pmod{49}$, that is, $v_p(f(x,y))  \equiv 0 \pmod{7}$. This shows that there exists an integer $k$, divisible at most by the primes $2$ and $7$, such that $f(x,y) = k \cdot z^7$ for some integer $z$. The above remarks imply that $k \mid 2^3 \cdot 7$ and $(z, 14)=1$. Note in particular that this implies $v_7(f(x,y))=v_7(k)$ and $v_2(f(x,y))=v_2(k)$.
   
    The splitting field of the dehomogenisation of $f$ is $K=\Q(\zeta_7)^+$, which is a totally real cubic field of class number $1$. In particular, there exists $\alpha \in \OK$ (a root of $f$) such that $f(x,y)$ is the norm of $x-\alpha y$. As $2$ is inert in $K$ we obtain $3 \mid v_2(f(x,y))=v_2(k)$. This implies $k \in \{1,7,2^3,2^3 \cdot 7\}$. 
    By \cite[Corollary 4.2]{furio24} the elliptic curve $E$ has potentially good reduction at $7$, so $v_7(j(E)) \ge 0$. If $7 \mid k$, this implies $0 \le v_7(j(E)) = 3v_7(g(x,y)) - 7v_7(f(x,y)) = 3v_7(g(x,y)) - 7$, and hence $v_7(j(E)) \ge 2$.
    If $X=X_{ns}^+(49)$, we know by \cite[Lemma 6.3]{furio24} that $E$ has potentially supersingular reduction at $7$, and by \Cref{prop: j and Hasse} we see that $v_7(j(E)) = 0$. This implies that in this case $7 \nmid k$.
    If instead $X=X_{ns}^\#(49)$, by \Cref{lemma: 7-adic valuation on X_ns^++(49)} we know that $v_7(j(E)) > 0$, and this can happen only if $7 \mid k$ (to see this, notice that the numerator $g(x,y)^3$ of $j(E)$ must be divisible by $7$, and the congruence $g(x,y) \equiv 0 \pmod{7}$ implies $x \equiv 0 \pmod 7$. The equation $f(x,y) = kz^7$ with $v_7(z)=0$ then gives $7 \mid k$).

    We can repeat the argument for the curve $X=X_{sp}^\#(49)$, which admits a map to $X_{sp}^+(7)$. By \Cref{thm: j-invariants Zywina}, an elliptic curve $E/\Q$ corresponding to a rational point on $X$ has $j$-invariant
    \[
    j(E) = \frac{xg_{sp}(x,y)^3}{(yf_{sp}(x,y))^7},
    \]
    where now
    \[
    g_{sp}(x,y) = (x+y)(x^2-5xy+y^2)(x^2-5xy+8y^2)(x^4-5x^3y+8x^2y^2-7xy^3+7y^4),
    \]
    \[
    f_{sp}(x,y) = x^3-4x^2y+3xy^2+y^3,
    \]
    and $x, y$ are relatively prime integers. One checks that $\gcd(xg_{sp}(x,y), y) = 1$ and $\gcd(x, yf_{sp}(x,y)) = 1$,    
    hence $\gcd(xg_{sp}(x,y), yf_{sp}(x,y)) = \gcd(g_{sp}(x,y),f_{sp}(x,y))$.  Computing the resultant of $f_{sp}(x,y)$ and $g_{sp}(x,y)$ as polynomials in $\Z[y]$, we obtain that $\gcd(f_{sp}(x,y), g_{sp}(x,y)) \mid 7^7 y^{27}$. However, we know that $\gcd(g_{sp}(x,y), y) = 1$, and so $\gcd(f_{sp}(x,y), g_{sp}(x,y)) \mid 7^7$.
    As above, we check that $f_{sp}(x,y) \equiv 0 \pmod{49}$ implies $x \equiv y \equiv 0 \pmod{7}$, hence $49 \nmid f_{sp}(x,y)$ and therefore $\gcd(f_{sp}(x,y), g_{sp}(x,y)) \mid 7$. In particular, applying \Cref{prop: ciuffetti denominators} and reasoning as above, we find that there exist integers $z$ and $w$ such that $f_{sp}(x,y) = k \cdot z^7$, with $k \in \{1,7\}$ and $y=w^7$ (this uses that $y$ and $f_{sp}(x,y)$ are coprime).
\end{proof}

The equations of \Cref{prop: norm equations} are \textit{a priori} weaker constraints on $j(E)$ with respect to the existence of a rational point with that $j$-invariant on the modular curves $X_{ns}^+(49), X_{ns}^\#(49)$ and $X_{sp}^\#(49)$. Indeed, when considering these equations we are simply looking for elliptic curves $E$ such that $\operatorname{Im}\rho_{E,7}$ is contained in either $C_{ns}^+(7)$ or $C_{sp}^+(7)$ and the denominator of $j(E)$ is a $49$-th power.
However, Darmon and Granville (\cite[Theorem 1]{darmon-granville95}) showed that equations of this kind admit a finite number of solutions in the integers, so (just like with rational points on the relevant modular curves) we are still dealing with a problem with a finite number of solutions.

The method used by Darmon and Granville to prove \cite[Theorem 1]{darmon-granville95} is not completely explicit. However, in specific cases, it is possible to relate the solutions of equations of the form $f(x,y) = kz^n$ with those of certain generalised Fermat equations. The result of the next proposition is closely related to \cite{MR2999040}, in which equations of the form $F(x,y) = k z^\ell$ are also studied. In particular, the identity $a^2+4b^3 = -27D k^2c^n$ in the proposition follows immediately from considering invariants and covariants of cubic forms, as in \cite[Equation (4)]{MR2999040}. We have chosen to provide a direct arithmetic proof both to offer a slightly different perspective and because this makes it easier to control certain arithmetic properties of the integers involved in our construction.

\begin{proposition}\label{prop: from norm to Fermat}
    Let $f(t) \in \Z[t]$ be a monic separable polynomial of degree $3$ and let $K$ be its splitting field. Suppose that $[K:\Q] \in \{1,3\}$. Let $F(x,y)$ be the homogenisation of $f$ and let $D$ be the discriminant of $f$.
    Fix $k \in \Z$ and $n \in \N$. Every integer solution $(x,y,z)$ to the equation $F(x,y) = kz^n$ gives rise to an integer solution of the equation $$a^2 + 4b^3 = -27Dk^2 c^n$$
    with $c=z^2$. This solution can be taken to satisfy the following additional properties:
    \begin{enumerate}
        \item Suppose that $[K:\Q]=3$ and let $p$ be a rational prime inert in $K$. If $(x,y,z) \ne (0,0,0)$, then $a \ne 0$ and $v_p(a) \equiv 0 \pmod 3$.
        \item Suppose $(x,y)=1$: then every prime $p$ dividing the gcd of at least two among $a,b,c$ divides $6Dk$.
    \end{enumerate}    
\end{proposition}

\begin{proof}
    First, we associate the trivial solution $(x,y,z) = (0,0,0)$ with the trivial solution $(a,b,c) = (0,0,0)$. From now on, we assume $(x,y,z) \ne (0,0,0)$.
    Consider the factorisation $F(x,y) = (x-\alpha_1 y)(x-\alpha_2y)(x-\alpha_3y)$ in $K$. As $(\alpha_i - \alpha_j)x = \alpha_i(x-\alpha_jy) - \alpha_j(x-\alpha_iy)$, we can write
    \begin{align*}
        0 &= (\alpha_1 - \alpha_2)x + (\alpha_2 - \alpha_3)x + (\alpha_3 - \alpha_1)x \\
        &= (\alpha_3-\alpha_2)(x-\alpha_1y) + (\alpha_1-\alpha_3)(x-\alpha_2y) + (\alpha_2-\alpha_1)(x-\alpha_3y).
    \end{align*}
    Since the Galois group of $K/\Q$ is contained in $A_3$, the discriminant $D$ of $f$ is a square. We can hence define $d=\sqrt{D}$. Set $\beta_i = (\alpha_{i-1}-\alpha_{i+1})(x-\alpha_iy)$, where indices are considered modulo $3$. We notice that for every $i$ we have $\Q(\beta_i) \subseteq K$, that $\beta_1 + \beta_2 + \beta_3 = 0$, and that $\beta_1 \beta_2 \beta_3 = \pm dkz^n$.
    Moreover, the equality $\Q(\beta_i)=K$ holds for every $i$. Indeed, if $K=\Q$ the statement is obvious, while in the case $[K:\Q]=3$ it suffices to prove that $\beta_i \notin \Q$. Suppose that $\beta_i \in \Q$. By definition, the numbers $\beta_1, \beta_2, \beta_3$ are conjugate under the action of $\Gal(K/\Q)$. In particular, this means that $\beta_1=\beta_2=\beta_3=0$, where the last equation follows from the equality $0 = \beta_1 + \beta_2 + \beta_3 = 3\beta_i$. Since $f$ is separable, we know that $\alpha_{i-1}-\alpha_{i+1} \ne 0$, and so $x-\alpha_iy=0$. However, for every $\beta_j$ with $j \ne i$ we conclude in the same way that $x-\alpha_jy = 0$, and if $y \ne 0$ this implies that $\alpha_i = \alpha_j$, giving a contradiction. If instead $y=0$, we obtain that $0=x-\alpha_iy= x$, and this gives the trivial solution $(x,y,z)=(0,0,0)$, which we already excluded.
 
    Define the polynomial $q(T) = (T-\beta_1)(T-\beta_2)(T-\beta_3)$. Expanding, we find $q(T) = T^3 +bT \mp dkz^n$, where $b=\beta_1\beta_2+\beta_2\beta_3+\beta_3\beta_1$ is an integer. The splitting field of $q$ is contained in $K$, hence the discriminant of $q$ is a square. In particular, there exists an integer $a=(\beta_1-\beta_2)(\beta_2-\beta_3)(\beta_3-\beta_1)$ such that $a^2 = -4b^3 -27d^2k^2z^{2n}$, which is the equation in the statement of the theorem.
    We now show that the solution thus constructed satisfies the additional properties given in the statement:
    \begin{enumerate}
        \item Since $(x,y,z) \ne (0,0,0)$, we know that the splitting field of $q(T)$ is $K$, so $q$ is separable (because it is irreducible), which implies that $a \ne 0$.
        Suppose that $v_p(a) > 0$. By assumption, $p$ is inert in $K$. Since $p \mid a = (\beta_1-\beta_2)(\beta_2-\beta_3)(\beta_3-\beta_1)$ and $p$ is prime in $\mathcal{O}_K$, we have $v_p(\beta_i - \beta_{i+1})>0$ for some $i \in \{1,2,3\}$. However, all $\beta_i-\beta_{i+1}$ are conjugate under the action of Galois, so every factor $\beta_i - \beta_{i+1}$ has the same valuation. It follows as desired that $v_p(a) = v_p(\beta_1-\beta_2) + v_p(\beta_2-\beta_3) + v_p(\beta_3-\beta_1) = 3 v_p(\beta_1 - \beta_2)$.
        \item Suppose that $p$ is a prime such that $p \nmid 6Dk$. It is easy to notice that if $p$ divides two among $a,b,c$, then it also divides the third. Since $p \mid c = z^2$, there exists a prime $\pi_i$ of $K$ above $p$ dividing $\beta_i$, and so $\pi_i \mid x-\alpha_iy$, because $p \nmid D$. We have 
        $$\gcd(x-\alpha_iy, x-\alpha_jy) \mid \gcd((\alpha_i-\alpha_j)x, (\alpha_i-\alpha_j)y) = \alpha_i - \alpha_j \mid D,$$
        and so $\pi_i \nmid \beta_{i-1}\beta_{i+1}$. This implies that $b \equiv \beta_{i-1}\beta_{i+1} \not\equiv 0 \pmod {\pi_i}$, which contradicts our assumption $p \mid b$. \qedhere
    \end{enumerate}    
\end{proof}

\begin{corollary}\label{cor: specific norm equations}
    Let $k \in \{1,8\}$ and let $n$ be a positive integer. 
    \begin{enumerate}
        \item For every triple $(x,y,z)$ of integers satisfying one of the equations
        \begin{align*}
            x^3-7x^2y+7xy^2+7y^3 = k \cdot z^n \qquad \text{or} \qquad x^3-4x^2y+3xy^2+y^3 = z^n,
        \end{align*}
        there exists an integer solution $(a,b,c)$ of the equation $a^2 + 28b^3 = -27c^n$ with $c=z^2$ which satisfies:
        \begin{enumerate}
            \item if $(x,y) \ne (0,0)$ and $p \not\equiv \pm 1 \pmod 7$ is a prime, then $v_p(a) \equiv 0 \pmod 3$;
            \item if $x,y$ are coprime and $n>4$, then $c$ is coprime with $42ab$;
        \end{enumerate}
        
        \item For every triple $(x,y,z)$ of integers satisfying one of the equations
        \begin{align*}
            x^3-7x^2y+7xy^2+7y^3 = 7k \cdot z^n \qquad \text{or} \qquad x^3-4x^2y+3xy^2+y^3 = 7z^n,
        \end{align*}
        there exists an integer solution $(a,b,c)$ of the equation $a^2 + 196b^3 = -27c^n$ with $c=z^2$ which satisfies:
        \begin{enumerate}
            \item if $(x,y) \ne (0,0)$ and $p \not\equiv \pm 1 \pmod 7$ is a prime, then $v_p(a) \equiv 0 \pmod 3$;
            \item if $x,y$ are coprime and $n>4$, then $c$ is coprime with $42ab$;
        \end{enumerate}
    \end{enumerate}
\end{corollary}

\begin{proof}
    We start by proving the first statement. First, in the case $k=8$ and $f_1(x,y) = x^3-7x^2y+7xy^2+7y^3$, we notice that $x,y$ must be either both even or both odd. If they are both even, we can write $(2x,2y)$ in place of $(x,y)$ and we reduce to the case $k=1$. If they are both odd, we can write $x=2w+y$ for some integer $w$, and so
    \begin{align*}
        f_1(2w+y,y) &= (2w+y)^3 - 7(2w+y)^2y + 7(2w+y)y^2 + 7y^3 = 8(w^3 - 2w^2y - wy^2 + y^3),
    \end{align*}
    and so we obtain the equation $f_2(w,y) = w^3 - 2w^2y - wy^2 + y^3 = z^n$. We can therefore assume that $k=1$ and prove the statement for the three equations
    \begin{align*}
        f_1(x,y) &:= x^3-7x^2y+7xy^2+7y^3 = z^n, \\
        f_2(x,y) &:= x^3 - 2x^2y - xy^2 + y^3 = z^n, \\
        f_3(x,y) &:= x^3-4x^2y+3xy^2+y^3 = z^n.
    \end{align*}
    The splitting field of each of the polynomials $f_i(x,1)$ is $K=\mathbb{Q}(\zeta_7)^+$, in which every prime $p \not\equiv \pm 1 \pmod 7$ is inert.
    We now compute the discriminant of the polynomials $f_1(x,1), \; f_2(x,1), \; f_3(x,1)$, which are respectively $2^6 \cdot 7^2$, $7^2$, and $7^2$.
    By Proposition \ref{prop: from norm to Fermat} we know that every solution $(x,y,z)$ gives rise to a triple $(a_0,b_0,c)$ such that $a_0^2+4b_0^3 = -2^6 \cdot 3^3 \cdot 7^2 \cdot c^n$ in the case of $f_1$, or $a_0^2+4b_0^3 = -3^3 \cdot 7^2 \cdot c^n$ in the cases of $f_2, f_3$, with $c=z^2$ in all cases.
    Since $p \not\equiv \pm 1 \pmod 7$ is inert in $K$, we can also assume $v_p(a_0) \equiv 0 \pmod 3$.
    Write $f$ for $f_1, f_2, f_3$, according to which equation we are solving.
    Since $7$ is totally ramified in the splitting field $K=\Q(\zeta_7)^+$ of $f(x,1)$, we can write $7 = \pi^3$ for a prime $\pi$ of the ring of integers of $\mathcal{O}_K$. We then have $\pi \mid \alpha_i - \alpha_j$ for every pair $\alpha_i, \alpha_j$ of roots of $f$. In particular, $\pi$ divides the values $\beta_i = (\alpha_{i-1}-\alpha_{i+1})(x-\alpha_iy)$, and by the proof of \Cref{prop: from norm to Fermat} we have $7 \mid b_0$, since $b_0=\beta_1\beta_2 + \beta_2\beta_3 + \beta_1\beta_3$. Similarly, in the case of $f_1$ we have $4 \mid b_0$: we have $2 \mid (\alpha_{i-1}-\alpha_{i+1}) \mid \beta_i$ for every $i$ (note that $2$ is prime in $K$), hence $4 \mid b_0=\beta_1\beta_2+\beta_2\beta_3+\beta_3\beta_1$. By congruence arguments, one shows easily that in the case of $f_1$ there exist integers $a,b$ such that $a_0 = 2^3 \cdot 7 \cdot a$ and $b_0 = 2^2 \cdot 7 \cdot b$, and so $a^2+28b^3 = -27c^n$.
    In the cases $f_2, f_3$ we find instead $a_0 = 7 \cdot a$ and $b_0 = 7 \cdot b$, and so we obtain again $a^2+28b^3 = -27c^n$.
    Note that in all cases the $p$-adic valuation of $a$ is the same as the $p$-adic valuation of $a_0$ for every prime $p \not\equiv \pm 1 \pmod 7$, hence it is congruent to $0$ modulo $3$.\\
    The second statement is proved similarly. As in the first part, we reduce to working with the three equations $f_i(x,y) = 7z^n$ for $i=1, 2, 3$. As above, we have solutions $(a_0, b_0, c)$ to the equation $a_0^2+4b_0^3 = -2^6 \cdot 3^3 \cdot 7^4  \cdot c^n$ in the case of $f_1$, or $a_0^2+4b_0^3 = -3^3 \cdot 7^4 \cdot c^n$ in the cases of $f_2, f_3$, with $c=z^2$ in all cases. By the same argument as above, we have $7 \mid b_0$, and in the case $i=1$  also $4 \mid b_0$. Hence, there exist integers $a_1,b_1$ such that $a_1^2 + 28b_1^3 = -27 \cdot 7^2 \cdot c^n$ in all cases. More precisely, we have
    \[
    a_0=2^3 \cdot 7 \cdot a_1, \ b_0=2^2 \cdot 7 \cdot b_1 \quad \text{ for }i=1 \qquad \text {and} \qquad a_0=7a_1, \ b_0=7b_1 \quad \text{ for }i=2,3.
    \]
    We easily see that $7$ must divide $a_1$, and so that $7 \mid b_1$. Writing $a_1=7a$ and $b_1=7b$ we obtain $a^2+196b^3= -27c^n$ as desired.
    Moreover, we again have that $v_p(a)=v_p(a_1)=v_p(a_0)$ is congruent to $0$ modulo $3$ for every $p \not\equiv \pm 1 \pmod 7$. \\
    We now show that if $x$ and $y$ are coprime, then $c$ is coprime with $42ab$.
    If $p$ is a prime dividing $(c, 42ab)$, by \Cref{prop: from norm to Fermat} we have $p \in \{2,3,7\}$. However, one can check that for all coprime integers $x,y$ and for every $i=1,2,3$ we have $f_i(x,y) \not\equiv 0$ modulo $16,3,49$, and since $n>4$ this gives $p \nmid z$ (which in turn implies $p \nmid c$).
\end{proof}

We now focus on the case $n=7$, which is the one that arises from \Cref{prop: norm equations}: we will consider the equations $a^2+28b^3=27c^7$ and $a^2+196b^3=27c^7$, where $c=-z^2$.
We will often need to refer to the arithmetic conditions in the previous statement, so we introduce the following definition for ease of reference:

\begin{definition}
    We say that a solution $(a, b, c)$ to one of the equations $a^2+28b^3=27c^7$ or $a^2+196b^3=27c^7$ satisfies $(\star)$ if the following conditions hold:
    \[
    (\star) \qquad  \begin{cases} \; \begin{array}{l}
        (c, 42ab) = 1 \\
        v_3(a)=0 \text{ or } v_3(a) \geq 3
    \end{array}
    \end{cases}
    \]
    We will often write that $(a,b,c)$ is a $(\star)$-solution.
\end{definition}

Note in particular that $(\star)$ implies $(a,b,c)=1$, so any $(\star)$-solution is primitive. Moreover, the first condition in $(\star)$ gives that $c$ is odd, which (together with \eqref{eq: Cns} or \eqref{eq: C sharp}) shows that $a$ is automatically odd. We will freely use these facts in the rest of the paper. We will show:

\begin{theorem}\label{thm: solutions twisted Fermat}\phantom{~}
    \begin{enumerate}
        \item The $(\star)$-solutions of the equation
        \begin{equation}\label{eq: Cns}
            a^2+28b^3=27c^7
        \end{equation}
        are precisely
        \[
        (\pm 1,-1, -1), \quad (\pm 27, -3, -1), \quad (\pm 2521, -61, -1).
        \]
        \item Assume \Cref{conj: XE3}. The $(\star)$-solutions of the equation
        \begin{equation}\label{eq: C sharp}
            a^2+196b^3=27c^7
        \end{equation}
        are precisely $(\pm 13, 196, -1)$.
    \end{enumerate}
\end{theorem}

\begin{remark}\label{rmk: small solutions}
There exist primitive solutions of Equations \eqref{eq: Cns} and \eqref{eq: C sharp} that are not $(\star)$-solutions. A computer search for primitive solutions to $a^2 + 28b^3 = 27c^7$ with $c \neq 0$, $|c| \leq 300$ found
\[
(\pm1, -1, -1), \quad (\pm27, -3, -1), \quad (\pm 2521, -61, -1), \quad (\pm 2\cdot 181 \cdot 313 \cdot 317, \, 3593, \, 90),
\]
the last of which is not a $(\star)$-solution. To find these solutions, we looped over $c \in [-300, 300] \cap \mathbb{Z}$ and computed, for each fixed $c$, the integral points on the elliptic curve $a^2+28b^3=27c^7$. In particular, the above list contains all primitive solutions with $|c| \leq 300$.
We also record the solution
\[
(\pm 2^{13} \cdot 5 \cdot 59957, -2^8 \cdot 1867, 2^4 \cdot 17)
\]
because, while not primitive, it corresponds to a rational point on one of the genus-3 curves we will have to study in \Cref{sect: genus 3}. A similar search for primitve solutions of $a^2+196b^3 = 27c^7$ found $(\pm 13, -1, -1)$ and $(\pm 16074, 39, 10)$. The second one is not a $(\star)$-solution.
\end{remark}

We can use \Cref{thm: solutions twisted Fermat} to determine the elliptic curves over $\mathbb{Q}$ having mod-49 representation contained in $C_{ns}^+(49)$. Assuming in addition \Cref{conj: XE3}, we can do the same for the exotic groups $G_{ns}^\#(49), G_{sp}^\#(49)$:

\begin{corollary}\label{cor: 7-adic classification}
    Let $G \in \{C_{ns}^+(49), G_{ns}^\#(49), G_{sp}^\#(49)\}$ and let $E/\Q$ be an elliptic curve with $\operatorname{Im}\rho_{E, 49}$ contained in $G$ up to conjugacy.
    \begin{enumerate}
        \item If $G=C_{ns}^+(49)$, then
        \[
        j(E) \in \{ -2^{18} 3^3 5^3 23^3 29^3,\; -2^{15} 3^3 5^3 11^3,\;  -2^{18} 3^3 5^3 11^3,\;  -2^{15}, \; 2^6 3^3, \; 2^6 5^3, \; 2^3 3^3 11^3 \}.
        \]
        In all cases, $E$ has complex multiplication.
        \item Assume \Cref{conj: XE3}. If $G=G_{ns}^\#(49)$ or $G=G_{sp}^\#(49)$, then $j(E)=0$ and hence $E$ has complex multiplication.
    \end{enumerate}
\end{corollary}

\begin{remark}
    The $j$-invariants in case 1 correspond to elliptic curves having complex multiplication respectively by the rings
    \[
    \mathbb{Z}\left[ \frac{1+\sqrt{-163}}{2} \right], \;\; \mathbb{Z}\left[ \frac{1+\sqrt{-67}}{2} \right], \;\; \mathbb{Z}\left[ \frac{1+\sqrt{-43}}{2} \right], \;\; \mathbb{Z}\left[ \frac{1+\sqrt{-11}}{2} \right], \;\; \mathbb{Z}[i], \;\; \mathbb{Z}[\sqrt{-2}], \;\; \mathbb{Z}[2i].
    \]
    The prime $7$ is inert in these rings, so the general theory of complex multiplication shows that the mod-49 representation of an elliptic curve with CM by these rings does in fact land inside the normaliser of a non-split Cartan subgroup.

    For the cases in (2), we read from \cite[\href{https://beta.lmfdb.org/ModularCurve/Q/49.147.9.a.1/}{modular curve 49.147.9.a.1}]{lmfdb} and \cite[\href{https://beta.lmfdb.org/ModularCurve/Q/49.196.9.a.1/}{modular curve 49.196.9.a.1}]{lmfdb} that the relevant modular curves possess a rational point with $j=0$.
\end{remark}

\begin{proof}
An elliptic curve as in the statement leads to a solution $(x,y,z)$ of one of the equations in the statement of \Cref{prop: norm equations}, with $(x,y)=1$. By \Cref{cor: specific norm equations}, to any such solution we can attach a $(\star)$-solution $(a,b,c)$ of \Cref{eq: Cns} or \Cref{eq: C sharp} with $c=-z^2$. By \Cref{thm: solutions twisted Fermat} we then have $c=-1$ and therefore $z=\pm 1$. As $(x, y, z)$ is a solution to one of the equations in \Cref{prop: norm equations} if and only if $(-x, -y, -z)$ is, and since $j(E)$ depends only on $x/y$, we can assume $z=1$. We are then left to solve a small number of Thue equations (those of \Cref{prop: norm equations} with $z=1$), which we can easily do in MAGMA. For every solution we then obtain a candidate for the $j$-invariant of $E$. Finally, for each candidate $j$-invariant $j_0$ we construct an elliptic curve $E_{j_0}$ over $\Q$ with $j(E_{j_0})=j_0$ and check whether $E$ has CM. If it does not, using Zywina's algorithm \cite{zywina22} we compute the adelic image of the Galois representation associated with $E_{j_0}$, and find in each case that $|\operatorname{Im} \rho_{E, 49}| = |\operatorname{Im} \rho_{E, 7}| \cdot 7^4$. This is incompatible with the fact that $\operatorname{Im} \rho_{E, 49}$ is contained in $G$ by assumption (note for example that $7^4 \nmid \#G$).
This shows that the non-CM cases cannot arise, and leaves us with finitely many CM $j$-invariants.

More precisely, for $G=G_{ns}^\#(49)$ the only $j$-invariant left after this computation is $j(E)=0$, as desired. For $G=G_{sp}^{\#}(49)$, we are left with the following possibilities:
\[
j(E) \in \{0, \; 2^4 3^3 5^3, \; 2^{15} 3^3, \; 2^{15} \cdot 3 \cdot 5^3 \}.
\]
These $j$-invariants correspond to elliptic curves with complex multiplication respectively by the rings
\[
\mathbb{Z}[\zeta_3], \;\; \mathbb{Z}[\sqrt{-3}], \;\; \mathbb{Z}\left[ \frac{1+\sqrt{-19}}{2} \right], \;\; \mathbb{Z}[3\zeta_3].
\]
Note that $7$ is split in each of these rings. Suppose by contradiction $j(E) \neq 0$. Since $E$ does not have extra automorphisms in this case, using any of \cite[Theorem 1.5]{MR3766118}, \cite[Corollary 1.5]{MR4077686} or \cite{MR4467123} it is easy to show that $\operatorname{Im} \rho_{E, 7^2}$ is the full normaliser of a split Cartan modulo $49$, which is not contained in $G$, contradiction.
\end{proof}

\begin{remark}
Going through the proof of \Cref{prop: norm equations}, we see that $z^{49}$ is the denominator of the $j$-invariant of the elliptic curves corresponding to the solutions of the equation $F(x,y) = k \cdot z^7$. In particular, once we know that $z = \pm 1$ (which follows from $c=-1$), we obtain that the elliptic curve $E$ has integral $j$-invariant. For the non-split case, the integral points of the curve $X_{ns}^+(7)$ were determined by Kenku in \cite{kenku85}.
\end{remark}

We also deduce \Cref{thm: conjectural 7-adic images}.

\begin{proof}[Proof of Theorem \ref{thm: conjectural 7-adic images}]
    Cases (i) and (ii) are identical to the corresponding cases in \Cref{thm: rszb}.
    Consider cases (iii) and (iv) in \Cref{thm: rszb}. The subcases $p^n=3^3$ and $p^n=7^2$ of (iii) are impossible by \cite{balakrishnan25} and \Cref{thm: Cns+(49)}, respectively. Assuming \Cref{conj: XE3}, case (iv) does not arise by \Cref{cor: 7-adic classification}.
\end{proof}

The rest of the paper is dedicated to proving \Cref{thm: solutions twisted Fermat}.

\section{From Fermat equations to elliptic curves via modularity}\label{sect: Fermat to ell curves}

We will associate with every solution of Equations \eqref{eq: Cns} and \eqref{eq: C sharp} certain elliptic curves over $\mathbb{Q}$. As an intermediate step, we begin by introducing the following curves:
\begin{definition}\label{def: Eabc tilde Fabc tilde}
    Let $(a,b,c) \in \mathbb{Z}^3$ be a $(\star)$-solution of the equation $a^2+28b^3=27c^7$. We let $\tilde{E}_{(a,b,c)}/\mathbb{Q}$ be the elliptic curve
\begin{equation}\label{eq: Eabc tilde}
\tilde{E}_{(a,b,c)} : y^2 = x^3+3 \cdot 7 \cdot b \cdot x-7a.    
\end{equation}
Similarly, let $(a,b,c) \in \mathbb{Z}^3$ be a $(\star)$-solution of the equation $a^2+196b^3=27c^7$. We let $\tilde{F}_{(a,b,c)}/\mathbb{Q}$ be the elliptic curve
\begin{equation}\label{eq: Fabc tilde}
    \tilde{F}_{(a,b,c)} : y^2 = x^3+3 \cdot 7^2 \cdot b \cdot x-7^2a.
\end{equation}
\begin{remark}
    As we will see in the rest of this section, the $j$-invariants of these curves are proportional to $b^3/c^7$. Moreover, it is immediate to recover a $(\star)$-solution $(a,b,c)$ of one of the equations \eqref{eq: Cns} and \eqref{eq: C sharp} from the ratio $b^3/c^7$.
In particular, knowing the $j$-invariants of all the curves $\tilde{E}_{(a,b,c)}, \tilde{F}_{(a,b,c)}$ that arise from $(\star)$-solutions of Equations \eqref{eq: Cns} and \eqref{eq: C sharp} would be enough to prove \Cref{thm: solutions twisted Fermat}.
\end{remark}
\end{definition}

\begin{lemma}\label{lemma: discriminants Eabc Fabc tilde}
Let $(a,b,c)$ be a $(\star)$-solution to \Cref{eq: Cns} (resp.~\Cref{eq: C sharp}).
    The discriminant of the model in \Cref{eq: Eabc tilde} (resp.~in \Cref{eq: Fabc tilde}) is $-2^4 \cdot 3^6 \cdot 7^2 \cdot c^7$ (resp.~$-2^4 \cdot 3^6 \cdot 7^4 \cdot c^7$).
\end{lemma}
\begin{proof}
For \eqref{eq: Eabc tilde} we have $\Delta_{\tilde{E}_{(a,b,c)}} = -16(4 (21b)^3 + 27 (-7a)^2) = -2^4 \cdot 3^3 \cdot 7^2 (a^2+28b^3)$. By assumption, $a^2+28b^3=27c^7$, which gives the formula in the statement. The case of the elliptic curve given in \Cref{eq: Fabc tilde} is very similar.
\end{proof}

\begin{lemma}\label{lemma: twists of Eabc Fabc tilde}
    Let $(a,b,c)$ be a $(\star)$-solution to \eqref{eq: Cns} (resp.~\eqref{eq: C sharp}). The quadratic twist by $-3$ of $\tilde{E}_{(a,b,c)}$ (resp.~of $\tilde{F}_{(a,b,c)}$) has good reduction at $3$.
    More precisely, there exists an integral model of this twist with discriminant $-2^4 \cdot 7^2 \cdot c^7$ (resp.~$-2^4 \cdot 7^4 \cdot c^7$).
\end{lemma}
\begin{proof}
Denote by $\tilde{E}_{(a,b,c)}^{(-3)}$ the quadratic twist of $\tilde{E}_{(a,b,c)}$ by $-3$.
    A model for this curve is given by
    \[
    y^2 = x^3 + 21 \cdot (-3)^2 \cdot b x -7a \cdot (-3)^3.
    \]
    Suppose first that $3 \mid a$. By condition $(\star)$ we then have $v_3(a) \geq 3$, and taking \eqref{eq: Cns} modulo $3$ shows that $v_3(b) \geq 1$. It follows that $v_3(21 \cdot (-3)^2 \cdot b)\geq 1+2+1 = 4$ and $v_3(-7a \cdot (-3)^3) \geq 6$. Replacing $x \to 3^2w$ and $y \to 3^3z$ and dividing by $3^6$ we then find the integral model
    \[
    z^2 = w^3 + \frac{7b}{3} x + \frac{7a}{3^3}.
    \]
    Using the fact that $a^2+28b^3=27c^7$ by assumption, the discriminant of this model is
    \[
    -16 \left( 4\left(\frac{7b}{3} \right)^3 + 27 \left( \frac{7a}{3^3} \right)^2 \right) = -\frac{16}{27} \cdot 7^2 \cdot \left( 4 \cdot 7b^3 +  a^2 \right) = - \frac{16}{27} \cdot 7^2 \cdot 27c^7=-2^47^2c^7,
    \]
    and $(c,3)=1$ by $(\star)$. Thus, this is an integral model of $\tilde{E}_{(a,b,c)}^{(-3)}$ with good reduction at $3$.

    Suppose now that $v_3(a)=0$. The equation $a^2+28b^3=27c^7$ then also gives $3 \nmid b$. Elementary arguments using congruences modulo 3, 9 and 27 yield the following: we can write $a=9a_2+a_0$ with $a_0 \in \{-1, 1\}$, and $b=9b_2+3b_1-1$ with $3 \mid b_1a_0-a_2$.
    Using $a_0^2=1$, the change of variables $x \to 3^2w -a_0, y \to 3^3z$ (followed by division by $3^6$) gives
    \begin{equation}\label{eq: model of good reduction at 3}
        z^2 = w^3 - a_0w^2 + \frac{1+7b}{3}w + \frac{-a_0-21a_0b+7a}{27}.    
    \end{equation}
    Replacing $a=9a_2+a_0$, $b=9b_2+3b_1-1$ and using the congruence $3 \mid b_1a_0-a_2$, one sees immediately that this equation has integral coefficients.
    Its discriminant can be computed directly, or more easily as follows: twisting by $(-3)$ increases the discriminant by a factor $(-3)^6$ and the change of variables divides it by a factor $3^{12}$. It follows that the discriminant of the integral model \eqref{eq: model of good reduction at 3} is $3^{-6}$ times the discriminant of the model in \Cref{eq: Eabc tilde}, which is $-2^43^37^2(a^2+28b^3)=-2^43^67^2 c^7$. The discriminant of \eqref{eq: model of good reduction at 3} is therefore $-2^47^2c^7$, which is prime to 3. This shows that $\tilde{E}_{(a,b,c)}^{(-3)}$ has good reduction at $3$, as claimed.

    The case of the quadratic twist $\tilde{F}^{(-3)}_{(a,b,c)}$ of $\tilde{F}_{(a,b,c)}$ is very similar. If $v_3(a) \geq 3$, the same change of variables as in the previous case produces a model with good reduction at $3$. If $v_3(a)=0$, as above we write $a=9a_2+3a_1+a_0, b=9b_2+3b_1+b_0$ with $a_0, b_0 \in \{-1, 1\}$ and $a_1, b_1 \in \{-1, 0, 1\}$. We then consider \Cref{eq: C sharp} modulo 3, $3^2$, $3^3$ to find
    $b_0=-1$, $a_1=a_0$ and $b_1+1 \equiv a_0a_2 \pmod{3}$. The change of variables
    \[
    x \to 3^2 w - 3a_0, \quad y \to 3^3 z
    \]
    gives the model
    \[
    z^2 = w^3 -a_0 w^2 + \frac{1+49b}{3}w + \frac{-a_0 - 3\cdot 7^2\cdot a_0 \cdot b  + 49a}{27}.
    \]
    Using the congruences given above, one checks that this model is integral. Since its discriminant is prime to $3$, this shows as desired that $\tilde{F}_{(a,b,c)}^{(-3)}$ has good reduction at $3$.
\end{proof}

Motivated by the previous lemma, we replace $\tilde{E}_{(a,b,c)}$ and $\tilde{F}_{(a,b,c)}$ with their twists by $-3$ and introduce the following notation:
\begin{definition}\label{def: Eabc Fabc}
    Let $(a,b,c) \in \mathbb{Z}^3$ be a $(\star)$-solution of the equation $a^2+28b^3=27c^7$. We let $E_{(a,b,c)}/\mathbb{Q}$ be the elliptic curve
\begin{equation}\label{eq: Eabc}
E_{(a,b,c)} : y^2 = x^3+ 3\cdot 7 \cdot (-3)^2 \cdot b \cdot x -7 \cdot (-3)^3 a.
\end{equation}
Similarly, let $(a,b,c) \in \mathbb{Z}^3$ be a $(\star)$-solution of the equation $a^2+196b^3=27c^7$. We let $F_{(a,b,c)}/\mathbb{Q}$ be the elliptic curve
\begin{equation}\label{eq: Fabc}
    F_{(a,b,c)} : y^2 = x^3+3 \cdot 7^2 \cdot (-3)^2 \cdot b \cdot x-7^2 \cdot (-3)^3 a.
\end{equation}
\end{definition}
Note that these models are not minimal at 3: the corresponding minimal models are described in the proof of \Cref{lemma: twists of Eabc Fabc tilde}.

\begin{lemma}\label{lemma: basic properties Eabc Fabc}\phantom{~}
\begin{enumerate}
    \item Let $(a,b,c)$ be a $(\star)$-solution of the equation $a^2+28b^3=27c^7$. 
    \begin{enumerate}
        \item The elliptic curve $E_{(a,b,c)}$ has $j$-invariant $j(E_{(a,b,c)}) = 2^8 \cdot 7 \cdot \frac{b^3}{c^7}$.
        \item The elliptic curve $E_{(a,b,c)}$ has good reduction at all primes not dividing $14c$, multiplicative reduction at all primes $p \neq 2,7$ dividing $c$, and additive potentially good reduction at $2$ and $7$. Its conductor is $2^f 7^2 \prod_{\substack{p \not \in {2,7} \\ p \mid c}} p$ for some integer $f \in \{2,3,4\}$, and its minimal discriminant is $\Delta_{min} = - 2^4 \cdot 7^2 \cdot c^7$.
    \end{enumerate}
    
    \item Let $(a,b,c)$ be a $(\star)$-solution of the equation $a^2+196b^3=27c^7$.
    \begin{enumerate}
        \item The elliptic curve $F_{(a,b,c)}$ has $j$-invariant $j(F_{(a,b,c)}) = 2^8 \cdot 7^2 \cdot \frac{b^3}{c^7}$.
        \item The elliptic curve $F_{(a,b,c)}$ has good reduction at all primes not dividing $14c$, multiplicative reduction at all primes $p \neq 2, 7$ dividing $c$, and additive reduction at $2$ and $7$. Its conductor is $2^f 7^2 \prod_{\substack{p \not \in {2,7} \\ p \mid c}} p$ for some integer $f \in \{2,3,4\}$, and its minimal discriminant is $\Delta_{min} = -2^4 \cdot 7^4 \cdot c^7$.
    \end{enumerate}
\end{enumerate}
\end{lemma}

\begin{proof}
    Using \Cref{lemma: discriminants Eabc Fabc tilde} we compute $j(\tilde{E}_{(a,b,c)}) = -1728 \cdot \frac{(2^2 \cdot 3 \cdot 7 \cdot b)^3}{\Delta} = \frac{2^{12} \cdot 3^6 \cdot 7^3 \cdot b^3}{2^4 \cdot 3^6 \cdot 7^2 \cdot c^7} = 2^8 \cdot 7 \cdot \frac{b^3}{c^7}$, and similarly $j(\tilde{F}_{(a,b,c)}) = 2^8 \cdot 7^2 \cdot \frac{b^3}{c^7}$. Since the $j$-invariant does not change under quadratic twists, this proves the first part of each statement.
    
    The expressions for the discriminants of $E_{(a,b,c)}$ and $F_{(a,b,c)}$ given in \Cref{lemma: discriminants Eabc Fabc tilde} show that $E_{(a,b,c)}, F_{(a,b,c)}$ have good reduction at every prime not dividing $14c$. Moreover, for $p \in \{2,7\}$ we see that $v_p(j(E_{(a,b,c)}) \ge 0$, because condition $(\star)$ gives $p \nmid c$, and so $E_{(a,b,c)}$ has potentially good reduction at $p$. If instead $p \mid c$, by $(\star)$ we have $p \nmid 42ab$. The usual quantity $c_4$ \cite[\S 3.1]{MR2514094} associated with $E_{(a,b,c)}$ is equal to $-48 \cdot 21 \cdot (-3)^2 \cdot b$ and so $p \nmid c_4$. This implies that $E_{(a,b,c)}$ has multiplicative reduction at $p$ by \cite[Proposition VII.5.1]{MR2514094}. The same argument shows that if $p \mid c$ then $F_{(a,b,c)}$ has multiplicative reduction at $p$. \\
    We now observe that the discriminants of $E_{(a,b,c)}$ and $F_{(a,b,c)}$ given in \Cref{lemma: twists of Eabc Fabc tilde} are globally minimal: indeed, the exponents of $2$ and $7$ are smaller than $12$, and for every prime $p \mid c$ the curves $E_{(a,b,c)}$ and $F_{(a,b,c)}$ have multiplicative reduction at $p$, because $v_p(j) = - v_p(\Delta)$, and so $p \nmid c_4$. Let $N$ be the conductor of $E_{(a,b,c)}$. Since $E_{(a,b,c)}$ has additive reduction at $2$ and $7$, we have $7^2 \mid\mid N$ and $2^f \mid\mid N$ for some $f \ge 2$, while for $p \mid c$ we have $p \mid\mid N$, because $E_{(a,b,c)}$ has multiplicative reduction at $p$. By the Ogg--Saito formula (see for example \cite[Formula IV.11.1]{silverman-advanced-topics}) we have $v_2(\Delta_{min}) = f + m-1$, where $m \ge 1$ is the number of components on the special fibre of (the minimal proper regular model of) $E_{\Q_2}$. Since $v_2(\Delta_{\min})=4$, we obtain $f \le 4$ as desired. The same argument applies to the conductor of $F_{(a,b,c)}$.
\end{proof}

\begin{lemma}\label{lemma: E7 irreducible}
    Let $(a,b,c)$ be a $(\star)$-solution of the equation $a^2+28b^3=27c^7$ (resp.~$a^2+196b^3=27c^7$). The natural Galois representation on $E_{(a,b,c)}[7]$ (resp.~$F_{(a,b,c)}[7]$) is irreducible.
\end{lemma}
\begin{proof}
Using the formulas for the $j$-invariant given in \Cref{lemma: basic properties Eabc Fabc} and that $2 \nmid c$, we see that the $j$-invariants of $E_{(a,b,c)}$ and $F_{(a,b,c)}$ have strictly positive $2$-adic valuation. The claim follows from \Cref{lemma: 7 isogenies}.
\end{proof}

\begin{lemma}\label{lemma: E7 inertia}\phantom{~}
\begin{enumerate}
    \item Let $(a,b,c)$ be a $(\star)$-solution of \Cref{eq: Cns} and let $E := E_{(a,b,c)}$ be the corresponding elliptic curve, as in \Cref{def: Eabc Fabc}. The semisimplification of the restriction of $\rho_{E,7}$ to an inertia group at $7$ is given by $\chi^{0} \oplus \chi^1$, where $\chi$ is the mod-$7$ cyclotomic character.
    \item Let $(a,b,c)$ be a $(\star)$-solution of \Cref{eq: C sharp} and let $F := F_{(a,b,c)}$ be the corresponding elliptic curve, as in \Cref{def: Eabc Fabc}. The semisimplification of the restriction of $\rho_{F,7}$ to an inertia group at $7$ is given by $\chi^{-1} \oplus \chi^2$, where $\chi$ is the mod-$7$ cyclotomic character.
\end{enumerate}

\end{lemma}
\begin{proof}
    By \Cref{lemma: basic properties Eabc Fabc} we know that $E$ has bad, potentially good reduction at 7; by \Cref{prop: j and Hasse}, the reduction is potentially ordinary since $j(E) \equiv 0 \pmod 7$ and $7 \equiv 1 \pmod{3}$.
    Moreover, $v_7(\Delta)=2$, and this valuation is clearly minimal. By \cite[Proposition 1 on page 9]{kraus97}, the semisimplification of the restriction of $E[7]$ to an inertia group at $7$ is $\chi^{1-\alpha} \oplus \chi^{\alpha}$, where $\alpha :=(p-1) \frac{v(\Delta_{\min})}{12} = 1$ in our case. The argument for $F$ is identical, except that now $\alpha = (p-1) \frac{v(\Delta_{\min})}{12}=2$.
\end{proof}

In the next proposition we use the modularity of elliptic curves over $\Q$ to reduce our problem to one about modular forms. Before stating it, we introduce some standard notation. Let $f$ be a weight-2 newform for $\Gamma_0(N)$ and let $\mathcal{O}_f := \mathbb{Z}[a_n(f) : n \in \mathbb{N}]$ be the $\Z$-algebra generated by the Fourier coefficients $a_n(f)$ of $f$. Let $\mathfrak{p}$ be a prime of $\mathcal{O}_f$ above $7$ and with residue field $\F_7$. To this data, one can attach a Galois representation with values in $\operatorname{GL}_2(\F_7)$ that we denote by $\rho_{f, \mathfrak{p}}$. Given an elliptic curve $E/\Q$, we say that the Galois representation $\rho_{E, 7}$ arises from the modular form $f$ and the prime $\mathfrak{p}$ (and write $\rho_{E, 7} \sim \rho_{f, \mathfrak{p}}$) if $a_\ell(f) \equiv a_\ell(E) \pmod{\mathfrak{p}}$ for all but finitely many primes $\ell$, where (for primes $\ell$ where $E$ has good reduction) we set $a_\ell(E) = \ell+1-\#E(\mathbb{F}_\ell)$. Note that the traces $a_\ell(E)$ for a density-one set of primes $\ell$ determine the semisimplification of $\rho_{E,7}$.

\begin{proposition}\label{prop: level lowering}\phantom{~}
\begin{enumerate}
    \item     Let $(a,b,c)$ be a $(\star)$-solution of the equation $a^2+28b^3=27c^7$. There exists $N' \in \{2^2 \cdot 7^2, 2^3 \cdot 7^2, 2^4 \cdot 7^2\}$, a newform $f$ of weight 2 for $\Gamma_0(N')$, and a prime $\lambda$ of the ring $\mathcal{O}_f=\mathbb{Z}[a_n(f) : n \in \mathbb{N}]$ with residue field $\mathbb{F}_7$ such that $\rho_{E_{(a,b,c)}, 7} \sim \rho_{f, \lambda}$.
    \item Let $(a,b,c)$ be a $(\star)$-solution of the equation $a^2+196b^3=27c^7$. There exists $N' \in \{2^2 \cdot 7^2, 2^3 \cdot 7^2, 2^4 \cdot 7^2\}$, a newform $f$ of weight 2 for $\Gamma_0(N')$, and a prime $\lambda$ of the ring $\mathcal{O}_f=\mathbb{Z}[a_n(f) : n \in \mathbb{N}]$ with residue field $\mathbb{F}_7$ such that $\rho_{F_{(a,b,c)}, 7} \sim \rho_{f, \lambda}$.
\end{enumerate}
\end{proposition}
\begin{proof}
Let $E:=E_{(a,b,c)}$ or $F_{(a,b,c)}$ according to which equation we are considering. Let $N$ be its conductor and $N' := N / \prod_{\substack{p \mid\mid N \\ v_p(\Delta_{\min})}} p$. By \Cref{lemma: basic properties Eabc Fabc} we obtain $N'=2^f 7^2$ for some $f \in \{2,3,4\}$. By \Cref{lemma: E7 irreducible} we know that $E[7]$ is irreducible. The result then follows from modularity \cite{MR1839918} (which shows that the Galois representation $E[7]$ arises from a weight 2 newform of level $\Gamma_0(N)$) and level lowering \cite{Ribet1, Ribet2} (which allows us to lower the level of the form to $\Gamma_0(N')$).
\end{proof}

We collect in the next three subsections some basic data about modular forms of weight $2$ and level $\Gamma_0(N)$ for $N \in \{2^2 \cdot 7^2, 2^3 \cdot 7^2, 2^4 \cdot 7^2\}$. The structure of the tables is as follows: we introduce a name for each Galois orbit of newforms of the given level, give the corresponding LMFDB label, the ring generated by the Fourier coefficients, the first terms in the $q$-expansion (enough to identify the forms uniquely), and -- when the coefficients are rational -- the LMFDB identifier of a corresponding elliptic curve $E/\Q$. In many cases, we will use the knowledge of these elliptic curves to prove certain properties of the associated mod-$7$ representations.

\subsection{Level $2^2 \cdot 7^2 = 196$}

\begin{proposition}[Level $2^2 \cdot 7^2 = 196$]\label{prop: forms of level 196}
    Let $N=2^2 \cdot 7^2$. There are three Galois orbits of newforms $f_1, f_2, f_3$ of weight $2$ and level $\Gamma_0(N)$, described in \Cref{table: level 196}. Let $\mathfrak{p}_1=(7, \beta-3), \mathfrak{p}_2=(7,\beta+3)$ be the two primes of $\frac{\mathbb{Z}[\beta]}{(\beta^2-2)}$ with residue field $\mathbb{F}_7$.
    The Galois representations $\rho_{f_1, 7}$ and $\rho_{f_2, 7}$ are quadratic twists of one another, while $\rho_{f_3, \mathfrak{p}_1}$ and $\rho_{f_3, \mathfrak{p}_2}$ are respectively isomorphic to $\rho_{f_1, 7}$ and $\rho_{f_2, 7}$.
\end{proposition}
\begin{proof}
From the LMFDB \cite{lmfdb} we read that $\rho_{f_1, 7}$ and $\rho_{f_2, 7}$ are quadratic twists of one another. To check that $\rho_{f_3, \mathfrak{p}_1}$ is isomorphic to $\rho_{f_1, 7}$ (resp.~that $\rho_{f_3, \mathfrak{p}_2}$ is isomorphic to $\rho_{f_2, 7}$) we apply Sturm's bound \cite[Theorem 1]{Sturm87}, checking the congruence of coefficients $a_n(f_1) \equiv a_n(f_3) \pmod{\mathfrak{p_1}}$ (and similarly for $f_2$ and $\mathfrak{p}_2$) for all $n$ up to $\lfloor \frac{km}{12} \rfloor=56$, where $k=2$ is the weight and $m = N \prod_{p \mid N} (1+\frac{1}{p}) = [\operatorname{SL}_2(\mathbb{Z}) : \Gamma_0(N)]$. The same congruence then holds for all $n$, which proves the desired isomorphism of representations.
\end{proof}

\begin{table}[h!]
\centering
\scriptsize
\begin{tabular}{@{}llllll@{}}
\toprule
Name & LMFDB Label & $\mathbb{Z}[a_n]$ & $q$-expansion & $E/\Q$ & See \\
\midrule
$f_1$ & \href{https://www.lmfdb.org/ModularForm/GL2/Q/holomorphic/196/2/a/a/}{196.2.a.a} & $\mathbb{Z}$ & $q - q^{3} - 3q^{5} - 2q^{9} - 3q^{11} - 2q^{13} + \cdots$ & \href{https://www.lmfdb.org/EllipticCurve/Q/196/a/2}{196.a2} & \Cref{prop: F cannot arise from f1 f2 f3 f4 f5 f6} \\
$f_2$ & \href{https://www.lmfdb.org/ModularForm/GL2/Q/holomorphic/196/2/a/b/}{196.2.a.b} & $\mathbb{Z}$ & $q + q^{3} + 3q^{5} - 2q^{9} - 3q^{11} + 2q^{13} + \cdots$ & \href{https://www.lmfdb.org/EllipticCurve/Q/196/b/2}{196.b2} \\
$f_3$ & \href{https://www.lmfdb.org/ModularForm/GL2/Q/holomorphic/196/2/a/c/}{196.2.a.c} & $\frac{\mathbb{Z}[\beta]}{(\beta^2-2)}$ & $q + 2\beta q^{3} - \beta q^{5} + 5q^{9} + 4q^{11} - 3\beta q^{13} + \cdots$ \\
\bottomrule
\end{tabular}
\caption{Modular forms data for level 196, weight 2}\label{table: level 196}
\end{table}

\subsection{Level $2^3 \cdot 7^2 = 392$}

The following result is proved exactly as \Cref{prop: forms of level 196}, namely, combining information from the LMFDB with Sturm's bound.
\begin{proposition}[Level $2^3 \cdot 7^2=392$]\label{prop: forms of level 392}
    Let $N=2^3 \cdot 7^2=392$. There are eight Galois orbits of newforms of weight $2$ and level $\Gamma_0(N)$, listed in \Cref{table: level 392}.
    \begin{enumerate}
        \item The two mod-$7$ representations arising from $f_6$ (corresponding to the two primes of norm $7$ in $\mathbb{Z}[a_n(f_6)]$) are isomorphic to $\rho_{f_4, 7}$ and $\rho_{f_5,7}$, which are quadratic twists of one another.
        \item The two mod-$7$ representations arising from $f_9$ (corresponding to the two primes of norm $7$ in $\mathbb{Z}[a_n(f_9)]$) are isomorphic to $\rho_{f_7, 7}$ and $\rho_{f_8,7}$, which are quadratic twists of one another.
    \end{enumerate}
\end{proposition}

\begin{table}[h!]
\centering
\scriptsize
\begin{tabular}{@{}llllll@{}}
\toprule
Name & LMFDB Label & $\mathbb{Z}[a_n]$ & $q$-expansion & $E/\Q$ & See  \\
\midrule
$f_4$ & \href{https://www.lmfdb.org/ModularForm/GL2/Q/holomorphic/392/2/a/a/}{392.2.a.a}	& $\mathbb{Z}$ & $q-3q^{3}+q^{5}+6q^{9}-q^{11}-2q^{13}+\cdots$  &  \href{https://www.lmfdb.org/EllipticCurve/Q/392/a/1}{392.a1} & \Cref{prop: F cannot arise from f1 f2 f3 f4 f5 f6}
\\
 $f_5$ & \href{https://www.lmfdb.org/ModularForm/GL2/Q/holomorphic/392/2/a/f/}{392.2.a.f}	& $\mathbb{Z}$ & $q+3q^{3}-q^{5}+6q^{9}-q^{11}+2q^{13}+\cdots$  & \href{https://www.lmfdb.org/EllipticCurve/Q/392/f/1}{392.f1}\\ 
 $f_6$ & \href{https://www.lmfdb.org/ModularForm/GL2/Q/holomorphic/392/2/a/g/}{392.2.a.g}	& $\frac{\mathbb{Z}[\beta]}{(\beta^2-2)}$ & $q+\beta q^{3}+2\beta q^{5}-q^{9}+6q^{11}-4\beta q^{13}+\cdots$ \\ \hline
$f_7$ & \href{https://www.lmfdb.org/ModularForm/GL2/Q/holomorphic/392/2/a/c/}{392.2.a.c} & $\mathbb{Z}$ & $q-q^{3}-q^{5}-2q^{9}+3q^{11}-6q^{13}+\cdots$ & \href{https://www.lmfdb.org/EllipticCurve/Q/392/c/1}{392.c1} & \Cref{prop: 392 2 a c}  \\
$f_8$ & \href{https://www.lmfdb.org/ModularForm/GL2/Q/holomorphic/392/2/a/e/}{392.2.a.e} & $\mathbb{Z}$ & $q+q^{3}+q^{5}-2q^{9}+3q^{11}+6q^{13}+\cdots$ & \href{https://www.lmfdb.org/EllipticCurve/Q/392/e/1}{392.e1} &   \\ 
 $f_9$ & \href{https://www.lmfdb.org/ModularForm/GL2/Q/holomorphic/392/2/a/h/}{392.2.a.h} & $\frac{\mathbb{Z}[\beta]}{(\beta^2-8)}$ & $q+\beta q^{3}+\beta q^{5}+5q^{9}-4q^{11}-\beta q^{13}+\cdots$ \\ \hline
$f_{10}$ & \href{https://www.lmfdb.org/ModularForm/GL2/Q/holomorphic/392/2/a/b/}{392.2.a.b}	& $\mathbb{Z}$ & $q-2q^{3}+4q^{5}+q^{9}-8q^{15}+2q^{17}+\cdots$  & \href{https://www.lmfdb.org/EllipticCurve/Q/392/b/1}{392.b1} & \Cref{prop: 392 2 a b and d} \\ \hline
 $f_{11}$ & \href{https://www.lmfdb.org/ModularForm/GL2/Q/holomorphic/392/2/a/d/}{392.2.a.d}	& $\mathbb{Z}$ & $q-2q^{5}-3q^{9}-4q^{11}-2q^{13}+6q^{17}+\cdots$ & \href{https://www.lmfdb.org/EllipticCurve/Q/392/d/1}{392.d1} & \Cref{prop: 392 2 a b and d}\\ 
\bottomrule
\end{tabular}
\caption{Modular forms data for level 392, weight 2}\label{table: level 392}
\end{table}

\subsection{Level $2^4 \cdot 7^2 = 784$}

Again in analogy with Propositions \ref{prop: forms of level 196} and \ref{prop: forms of level 392} we have:
\begin{proposition}[Level $2^4 \cdot 7^2=784$]\label{prop: forms of level 784}
    Let $N=2^4 \cdot 7^2=784$. There are nine Galois orbits of newforms of weight $2$ and level $\Gamma_0(N)$, listed in \Cref{table: level 784}.
    \begin{enumerate}[leftmargin=18pt]
        \item The two mod-7 representations arising from $f_{17}$ (corresponding to the two primes of norm $7$ in $\mathbb{Z}[a_n(f_{17})]$) are isomorphic to $\rho_{f_{15},7}$ and $\rho_{f_{16},7}$, which are quadratic twists of each other and of $\rho_{f_1, 7}$.
        \item The two mod-7 representations arising from $f_{20}$ (corresponding to the two primes of norm $7$ in $\mathbb{Z}[a_n(f_{20})]$) are isomorphic to $\rho_{f_{18},7}$ and $\rho_{f_{19},7}$, which are quadratic twists of each other and of $\rho_{f_4, 7}$.
        \item The two mod-7 representations arising from $f_{25}$ (corresponding to the two primes of norm $7$ in $\mathbb{Z}[a_n(f_{25})]$) are isomorphic to $\rho_{f_{23},7}$ and $\rho_{f_{24},7}$, which are quadratic twists of each other and of $\rho_{f_7, 7}$.
    \end{enumerate}
\end{proposition}

\begin{table}[h!]
\centering
\scriptsize
\begin{tabular}{@{}llllll@{}}
\toprule
Name & LMFDB Label & $\mathbb{Z}[a_n]$ & $q$-expansion & $E/\Q$  & See \\
\midrule
$f_{12}$ & \href{https://www.lmfdb.org/ModularForm/GL2/Q/holomorphic/784/2/a/b/}{784.2.a.b} & $\mathbb{Z}$ &$q - 2q^{3} + q^{9} + 4q^{13} - 6q^{17} + 2q^{19} + \cdots$ & \href{https://www.lmfdb.org/EllipticCurve/Q/784/b/5}{784.b5} & \Cref{prop: 392 2 a b and d} \\ \hline
$f_{13}$ & \href{https://www.lmfdb.org/ModularForm/GL2/Q/holomorphic/784/2/a/l/}{784.2.a.l} & $\frac{\mathbb{Z}[\beta]}{(\beta^2-2)}$ &$q + \beta q^{3} + 2\beta q^{5} - q^{9} + 2q^{11} + 4q^{15} + \cdots$ &  & \Cref{prop: 784.2.a.l} \\ \hline
$f_{14}$ & \href{https://www.lmfdb.org/ModularForm/GL2/Q/holomorphic/784/2/a/f/}{784.2.a.f} & $\mathbb{Z}$ &$q - 3q^{9} - 4q^{11} - 8q^{23} - 5q^{25} + 2q^{29} + \cdots$ & \href{https://www.lmfdb.org/EllipticCurve/Q/784/f/1}{784.f4} & \Cref{prop: 784.2.a.f} \\ \hline
$f_{15}$ & \href{https://www.lmfdb.org/ModularForm/GL2/Q/holomorphic/784/2/a/d/}{784.2.a.d} & $\mathbb{Z}$ &$q - q^{3} + 3q^{5} - 2q^{9} + 3q^{11} + 2q^{13} + \cdots$ & \href{https://www.lmfdb.org/EllipticCurve/Q/784/d/2}{784.d2}  & 196.2.a.a \\
$f_{16}$ &\href{https://www.lmfdb.org/ModularForm/GL2/Q/holomorphic/784/2/a/g/}{784.2.a.g} & $\mathbb{Z}$ &$q + q^{3} - 3q^{5} - 2q^{9} + 3q^{11} - 2q^{13} + \cdots$ & \href{https://www.lmfdb.org/EllipticCurve/Q/784/g/2}{784.g2} \\
$f_{17}$ & \href{https://www.lmfdb.org/ModularForm/GL2/Q/holomorphic/784/2/a/m/}{784.2.a.m} & $\frac{\mathbb{Z}[\beta]}{(\beta^2-2)}$ &$q + 2\beta q^{3} + \beta q^{5} + 5q^{9} - 4q^{11} + 3\beta q^{13} + \cdots$ \\ \hline
$f_{18}$ & \href{https://www.lmfdb.org/ModularForm/GL2/Q/holomorphic/784/2/a/a/}{784.2.a.a} & $\mathbb{Z}$ & $q - 3q^{3} - q^{5} + 6q^{9} + q^{11} + 2q^{13} + \cdots$ & \href{https://www.lmfdb.org/EllipticCurve/Q/784/a/1}{784.a1} & 392.2.a.a \\
$f_{19}$ &\href{https://www.lmfdb.org/ModularForm/GL2/Q/holomorphic/784/2/a/j/}{784.2.a.j} & $\mathbb{Z}$ & $q + 3q^{3} + q^{5} + 6q^{9} + q^{11} - 2q^{13} + \cdots$ & \href{https://www.lmfdb.org/EllipticCurve/Q/784/j/1}{784.j1} \\
$f_{20}$ &\href{https://www.lmfdb.org/ModularForm/GL2/Q/holomorphic/784/2/a/k/}{784.2.a.k} & $\frac{\mathbb{Z}[\beta]}{(\beta^2-2)}$ & $q + \beta q^{3} - 2\beta q^{5} - q^{9} - 6q^{11} + 4\beta q^{13} + \cdots$ \\ \hline
$f_{21}$ &\href{https://www.lmfdb.org/ModularForm/GL2/Q/holomorphic/784/2/a/i/}{784.2.a.i} & $\mathbb{Z}$ & $q + 2q^{3} + 4q^{5} + q^{9} + 8q^{15} + 2q^{17} + \cdots$ & \href{https://www.lmfdb.org/EllipticCurve/Q/784/i/2}{784.i2} & \Cref{prop: 392 2 a b and d}, 392.2.a.b \\ \hline
$f_{22}$ & \href{https://www.lmfdb.org/ModularForm/GL2/Q/holomorphic/784/2/a/e/}{784.2.a.e} & $\mathbb{Z}$ & $q - 2q^{5} - 3q^{9} + 4q^{11} - 2q^{13} + 6q^{17} + \cdots$ & \href{https://www.lmfdb.org/EllipticCurve/Q/784/e/4}{784.e4} & \Cref{prop: 392 2 a b and d}, 392.2.a.d \\ \hline
$f_{23}$ & \href{https://www.lmfdb.org/ModularForm/GL2/Q/holomorphic/784/2/a/c/}{784.2.a.c}  & $\mathbb{Z}$ &$q - q^{3} + q^{5} - 2q^{9} - 3q^{11} + 6q^{13} + \cdots$ & \href{https://www.lmfdb.org/EllipticCurve/Q/784/c/1}{784.c1} & \Cref{prop: 392 2 a c}, 392.2.a.c \\
$f_{24}$ & \href{https://www.lmfdb.org/ModularForm/GL2/Q/holomorphic/784/2/a/b/}{784.2.a.h} & $\mathbb{Z}$ &$q + q^{3} - q^{5} - 2q^{9} - 3q^{11} - 6q^{13} + \cdots$ & \href{https://www.lmfdb.org/EllipticCurve/Q/784/h/1}{784.h1} &  \\ 
$f_{25}$ & \href{https://www.lmfdb.org/ModularForm/GL2/Q/holomorphic/784/2/a/n/}{784.2.a.n}  & $\frac{\mathbb{Z}[\beta]}{(\beta^2-8)}$ &$q + \beta q^{3} - \beta q^{5} + 5q^{9} + 4q^{11} + \beta q^{13} + \cdots$ & & \\
\bottomrule
\end{tabular}
\caption{Modular forms data for level 784, weight 2}\label{table: level 784}
\end{table}

\subsection{Ruling out modular forms}

We now give several criteria that show the the mod-$7$ representation attached to a curve $E_{(a,b,c)}$ or $F_{(a,b,c)}$ cannot arise from specific modular forms.

\begin{proposition}\label{prop: 392 2 a b and d}
    Let $E'/\mathbb{Q}$ be an elliptic curve with additive, potentially multiplicative reduction at 7. 
    Let $(a,b,c)$ be a $(\star)$-solution of \Cref{eq: Cns} (resp.~\Cref{eq: C sharp}) and let $E := E_{(a,b,c)}$ (resp.~$E := F_{(a,b,c)}$) be the corresponding elliptic curve, as in \Cref{def: Eabc Fabc}. We have $E[7] \not \cong E'[7]$. In particular, we have $\rho_{E, 7} \not \sim \rho_{f_{i}, 7}$ for every $i \in \{10, 11, 12, 21, 22\}$.
\end{proposition}
\begin{proof}
    Let $I_7$ be an inertia group at $7$. By \Cref{lemma: E7 inertia}, the semisimplification of the restriction of $E[7]$ to $I_7$ is isomorphic to $\chi^0 \oplus \chi^1$.
    On the other hand, $E'$ has additive, potentially multiplicative reduction at 7. By \cite[Proposition 10 on page 26]{kraus97}, which applies since $E'$ has bad additive reduction, the restriction of $E'[7]$ to $I_7$ is isomorphic to $\chi^{\frac{p-1}{2}} \oplus \chi^{\frac{p+1}{2}} = \chi^3 \oplus \chi^4$. Since $\{\chi^0, \chi\} \neq \{\chi^3, \chi^4\}$, the restrictions of $E[7]$ and $E'[7]$ to $I_7$ are not isomorphic, hence a fortiori $E[7] \not \cong E'[7]$. The argument for $F$ is identical, using $\{\chi^{-1}, \chi^2\} \neq \{\chi^3, \chi^4\}$.
    The final statement follows because the mod-7 representations attached to $f_{10}, f_{11}, f_{12}, f_{22}$, and $f_{23}$ are the mod-7 representations given by the elliptic curves with LMFDB labels \href{https://www.lmfdb.org/EllipticCurve/Q/392/b/1}{392.b1}, \href{https://www.lmfdb.org/EllipticCurve/Q/392/d/1}{392.d1}, \href{https://www.lmfdb.org/EllipticCurve/Q/784/b/5}{784.b5}, \href{https://www.lmfdb.org/EllipticCurve/Q/784/i/2}{784.i2}, and \href{https://www.lmfdb.org/EllipticCurve/Q/784/e/4}{784.e4}, which have additive, potentially multiplicative reduction at $7$.
\end{proof}

\begin{proposition}\label{prop: 784.2.a.f}
    Let $E$ be one of the elliptic curves $E_{(a,b,c)}, F_{(a,b,c)}$ of \Cref{def: Eabc Fabc}. We have $\rho_{E,7} \not \sim \rho_{f_{14},7}$.
\end{proposition}
\begin{proof}
    The mod-7 representation $\rho_{f_{14},7}$ is isomorphic to the mod-7 representation of the elliptic curve \cite[\href{https://www.lmfdb.org/EllipticCurve/Q/784/f/4}{784.f4}]{lmfdb}, which has CM by the ring $\mathbb{Z}[\frac{1+\sqrt{-7}}{2}]$. Denote this elliptic curve by $E'$. One checks easily (either directly or using CM theory) that $E'$ admits a 7-isogeny defined over $\mathbb{Q}$, so $E'[7]$ (and hence $\rho_{f_{14},7}$) is reducible. Since we know from \Cref{lemma: E7 irreducible} that $\rho_{E,7}$ is irreducible, we cannot have $\rho_{E,7}  \sim \rho_{f_{14},7}$
\end{proof}

\begin{proposition}\label{prop: 392 2 a c}
    Let $(a,b,c)$ be a $(\star)$-solution of \Cref{eq: Cns} and let $E := E_{(a,b,c)}$ be the corresponding elliptic curve, as in \Cref{def: Eabc Fabc}. Let $E'/\mathbb{Q}$ be an elliptic curve having additive, potentially good reduction at $7$. If $v_7(\Delta_{E',\min}) \in \{4,10\}$, then $\rho_{E, 7} \not \sim \rho_{E', 7}$. In particular, we have $\rho_{E_{(a,b,c)}, 7} \not \cong \rho_{f_{i}, 7}$ for every $i \in \{7, 8, 23, 24\}$.
\end{proposition}
\begin{proof}
    Let $I_7$ be an inertia group at $7$. It suffices to show that $E[7] \not \cong E'[7]$ after restriction to $I_7$.
    By \Cref{lemma: E7 inertia}, the semisimplification of the restriction of $E[7]$ to $I_7$ is isomorphic to $\chi^0 \oplus \chi^1$.
    On the other hand, $v_7(\Delta_{\operatorname{min}}(E')) > 0$, so $E'$ has bad, potentially good reduction at 7; by \Cref{prop: j and Hasse}, the reduction is potentially ordinary since $j(E) \equiv 0 \pmod{7}$ and $7 \equiv 1 \pmod{3}$.
    By \cite[Proposition 1 on page 9]{kraus97}, the semisimplification of the restriction of $E'[7]$ to $I_7$ is isomorphic to $\chi^{\alpha} \oplus \chi^{1-\alpha}$ with $\alpha = (p-1) \cdot \frac{v(\Delta_{\min}(E'))}{12} = 2$ or $5$. Since $\{\chi^0, \chi^1\} \neq \{\chi^2, \chi^{-1}\}$ and $\{\chi^0, \chi^1\} \neq \{\chi^5, \chi^{-4}\}$ we are done.

    For the final statement, note that the mod-7 representations arising from the forms $f_7, f_8, f_{23}, f_{24}$ coincide with the mod-7 representations of the elliptic curves $E'$ with LMFDB labels \href{https://www.lmfdb.org/EllipticCurve/Q/392/c/1}{392.c1}, \href{https://www.lmfdb.org/EllipticCurve/Q/392/e/1}{392.e1}, \href{https://www.lmfdb.org/EllipticCurve/Q/784/c/1}{784.c1}, \href{https://www.lmfdb.org/EllipticCurve/Q/784/h/1}{784.h1}, which satisfy the assumptions of the proposition.
\end{proof}

We note explicitly that the previous proposition does not extend to show $\rho_{F_{(a,b,c)}, 7} \not \cong \rho_{f_7, 7}$ -- in fact, the non-trivial solution $(a,b,c) = (13, 196, -1)$ of \Cref{eq: C sharp} gives rise to an elliptic curve $F := F_{(a,b,c)}$ of conductor 392, corresponding to the modular form $f_7$ (see \cite[\href{https://www.lmfdb.org/EllipticCurve/Q/392/c/1}{elliptic curve 392.c1}]{lmfdb}). However, a very similar argument shows the following:
\begin{proposition}\label{prop: F cannot arise from f1 f2 f3 f4 f5 f6}
    Let $(a,b,c)$ be a $(\star)$-solution of \Cref{eq: C sharp} and let $F := F_{(a,b,c)}$ be the corresponding elliptic curve, as in \Cref{def: Eabc Fabc}. 
Let $E'/\mathbb{Q}$ be an elliptic curve having additive, potentially good reduction at $7$. If $v_7(\Delta_{E',\min}) \in \{2,8\}$, then $\rho_{F, 7} \not \sim \rho_{E', 7}$. In particular, we have $\rho_{F, 7} \not \cong \rho_{f_{i}, 7}$ for every $i \in \{1, 2, 4, 5 \}$.
\end{proposition}
\begin{proof}
    The curve $E'$ has bad, potentially good reduction at $7$ and its minimal discriminant has $7$-adic valuation $2$ or $8$. As in the proof of \Cref{prop: 392 2 a c}, this implies that the semisimplification of the restriction of $E'[7]$ to an inertia group at $7$ is isomorphic to $\chi^0 \oplus \chi^1$. On the other hand, the semisimplification of the restriction of $F[7]$ to an inertia group at $7$ is isomorphic to $\chi^{-1} \oplus \chi^{2}$ by \Cref{lemma: E7 inertia}. For the final statement, note that the mod-7 representations arising from $f_1, f_2, f_{4}, f_{5}$ coincide with the mod-7 representations of the elliptic curves $E'$ with LMFDB labels \href{https://www.lmfdb.org/EllipticCurve/Q/196/a/2}{196.a2}, \href{https://www.lmfdb.org/EllipticCurve/Q/196/b/2}{196.b2}, \href{https://www.lmfdb.org/EllipticCurve/Q/392/a/1}{392.a1}, \href{https://www.lmfdb.org/EllipticCurve/Q/392/f/1}{392.f1}, which satisfy the assumptions of the proposition.
\end{proof}

\begin{proposition}\label{prop: 784.2.a.l}
Let $(a,b,c)$ be a $(\star)$-solution of \Cref{eq: Cns} (resp.~\Cref{eq: C sharp}) and let $E := E_{(a,b,c)}$ (resp.~$E := F_{(a,b,c)}$) be the corresponding elliptic curve, as in \Cref{def: Eabc Fabc}. Let $f := f_{13}$.
Suppose that there exists a prime $\mathfrak{p}$ of $\mathcal{O}_f = \mathbb{Z}[\beta]/(\beta^2-2)$ with residue field $\mathbb{F}_7$ such that $\rho_{E, 7} \cong \rho_{f, \mathfrak{p}}$.
Then $E$ has complex multiplication and $c= \pm 1$.
\end{proposition}
\begin{proof}
    Consider the elliptic curve $E'/\mathbb{Q}(\sqrt{-7})$ with LMFDB label \href{https://www.lmfdb.org/EllipticCurve/2.0.7.1/12544.5/d/1}{12544.5-d1}. Let $K=\mathbb{Q}(\sqrt{-7})$. If $a \in K$ is a root of the polynomial $x^2-x+2$, a minimal equation for $E'$ is given by
    \[
    {y}^2={x}^{3}+\left(-a+1\right){x}^{2}+\left(37a-166\right){x}-148a+808.
    \]
    Given a representation $\rho$ with values in $\operatorname{GL}_2(\mathbb{F}_7)$, we denote by $\rho^{\operatorname{ss}}$ its semisimplification. We read from the LMFDB that $E'$ is $\mathbb{Q}$-curve with corresponding modular form $f_{13}$. In particular, we have the isomorphism of representations
    \[
    \rho_{f, \mathfrak{p}}^{\operatorname{ss}}|_{G_K} \cong \rho_{E'/K, 7}^{\operatorname{ss}},
    \]
    where $G_K := \Gal(\overline{K}/K)$.
    Factoring the 7-division polynomial of $E'$ over $K$ (as well as its resultant with the polynomial defining $E'$), we find that $E'[7]$ is reducible and $\#\rho_{E'/K, 7}(G_K) = [K(E'[7]) : K]=42$. Taking the semisimplification of $\rho_{E'/K, 7}$ then kills the elements of order 7, so we obtain
    \[
    \#\rho_{E'/K, 7}^{\operatorname{ss}}(G_K) \mid \frac{42}{7}=6.
    \]
    By assumption, the representation $\rho_{f, \mathfrak{p}} = \rho_{E,7}$, restricted to $G_K$, coincides (up to semisimplification) with $\rho_{E'/K, 7}$, which is reducible. 
    Set $G_\Q := \Gal(\overline{\Q}/\Q)$. There are two cases:
    \begin{itemize}
        \item Suppose that the order of $\rho_{E,7}(G_\mathbb{Q})$ is divisible by $7$. Since $E[7]$ is irreducible by \Cref{lemma: E7 irreducible}, the group $\rho_{E,7}(G_\mathbb{Q})$ is not contained in a Borel subgroup of $\operatorname{GL}_2(\mathbb{F}_7)$. Combining this fact with $7 \mid \#\rho_{E,7}(G_\mathbb{Q})$ and $\det \left(\rho_{E,7}(G_\mathbb{Q}) \right) = \mathbb{F}_7^\times$ and using the classification of the maximal subgroups of $\operatorname{GL}_2(\mathbb{F}_7)$, we obtain $\rho_{E,7}(G_\mathbb{Q}) = \operatorname{GL}_2(\mathbb{F}_7)$. Thus, $\rho_{E,7}(G_K)$ has index at most $2$ in $\operatorname{GL}_2(\mathbb{F}_7)$, and therefore acts irreducibly on $\mathbb{F}_7^2$, which contradicts the fact shown above that $\rho_{E'/K, 7}$ is reducible.
    
        \item Suppose that the order of $\rho_{E,7}(G_\mathbb{Q})$ is not divisible by $7$. Then $\rho_{E,7}$ is semisimple (even after restriction to the subgroup $G_K$), and irreducible by \Cref{lemma: E7 irreducible}. We then have
        \[
        \#\rho_{E,7}(G_K) = \#\rho_{E,7}^{\operatorname{ss}}(G_K) = \#\rho_{E'/K, 7}^{\operatorname{ss}}(G_K) \mid 6
        \]
        and therefore
        \[
        \#\rho_{E,7}(G_\Q) \mid 6 [K:\mathbb{Q}] = 12.
        \]
        In particular, $[\operatorname{GL}_2(\mathbb{F}_7) : \rho_{E,7}(G_\Q)]$ is divisible by $\frac{\#\operatorname{GL}_2(\mathbb{F}_7)}{12} =168$.
        However, \cite[Section 1.4]{zywina15} shows that if $E$ is non-CM, then the index $[\operatorname{GL}_2(\mathbb{F}_7) : \rho_{E,7}(G_\Q)]$ is not divisible by $168$. This contradiction shows as desired that $E$ is CM, hence that its $j$-invariant is an integer. The formula for the $j$-invariant of $E$ given in \Cref{lemma: basic properties Eabc Fabc}, combined with the fact that $(c, 42ab)=1$ (which holds because $(a,b,c)$ is a $(\star)$-solution), shows that $c^7 \mid 1$, hence $c = \pm 1$ as desired. \qedhere
    \end{itemize}
\end{proof}

\section{From elliptic curves with given $7$-torsion to rational points on genus-3 curves}\label{sect: ell curves to ratpts genus 3}

For every elliptic curve $E/\mathbb{Q}$ there is an affine modular curve $Y_E(7)$ whose $\Q$-rational points parametrise $\overline{\Q}$-isomorphism classes of pairs $(E', \varphi)$, where $E'$ is an elliptic curve over $\mathbb{Q}$ and $\varphi : E[7]\to E'[7]$ is a \textit{symplectic} isomorphism, that is, an isomorphism such that $\bigwedge ^2 \varphi : \bigwedge^2 E[7] \cong \mu_7 \to \bigwedge^2E'[7] \cong \mu_7$ is the identity (here the identification $\bigwedge^2 E[7] \cong \mu_7$ is given by the Weil pairing, and similarly for $E'$). This affine modular curve can be completed to a smooth projective curve over $\mathbb{Q}$ that we denote by $X_E(7)$. Similarly, there is an affine modular curve $Y_E^-(7)$ whose rational points parametrise pairs $(E', \varphi)$, where now $\varphi : E[7]\to E'[7]$ is an \textit{anti-symplectic} isomorphism, that is, an isomorphism such that $\bigwedge ^2 \varphi : \bigwedge^2 E[7] \cong \mu_7 \to \bigwedge^2E'[7] \cong \mu_7$ coincides with the map $x \mapsto x^{-1}$. This curve can also be completed to a smooth projective curve $X_E^{-}(7)$ over $\mathbb{Q}$. See also \cite[\S 4.4]{PSS07} for further details. For ease of notation, we will sometimes write $X_E^{+}(7)$ to denote $X_E(7)$, so that $X_E^{\pm}(7)$ denotes the pair of modular curves $\{X_E(7), X_E^-(7)\}$.

Our interest in this construction comes from the following lemma, which follows essentially tautologically from the above facts:
\begin{lemma}[{\cite[Lemma 4.4]{PSS07}}]\label{lemma: why we use XE7}
    Fix an elliptic curve $E/\Q$. Let $F/\Q$ be an elliptic curve such that $E[7] \cong F[7]$ as Galois modules (without symplectic structure). If $F'$ is a quadratic twist of $F$, then $j(F') \in X(1)(\Q)$ is in the image of $X_E(7)(\Q) \to X_1(\Q)$ or $X_E^-(7)(\Q) \to X_1(\Q)$.
\end{lemma}

Our analysis in the previous sections shows that, if $(a,b,c)$ is a $(\star)$-solution to \Cref{eq: Cns} or \eqref{eq: C sharp}, then the mod-7 representation of the corresponding elliptic curve $E_{(a,b,c)}$ or $F_{(a,b,c)}$ belongs to a short list of isomorphism classes, each of which is also the mod-7 representation of an elliptic curve over $\Q$. This will allow us to show that $j(E_{(a,b,c)})$, respectively $j(F_{(a,b,c)})$, lies in the image of finitely many maps $X_{E_i}(7)(\Q) \to X_1(\Q)$, where the $E_i/\Q$ are explicit elliptic curves. Specifically, we let 
\begin{itemize}
    \item $E_1 : y^2=x^3-x^2-2x+1 / \mathbb{Q}$ be the elliptic curve with LMFDB label \href{https://www.lmfdb.org/EllipticCurve/Q/196/a/2}{196.a2}. It corresponds to the modular form $f_1$ with LMFDB label \href{https://www.lmfdb.org/ModularForm/GL2/Q/holomorphic/196/2/a/a/}{196.2.a.a}.
    \item $E_2 : y^2 = x^3-7x+7/\mathbb{Q}$ be the elliptic curve with LMFDB label \href{https://www.lmfdb.org/EllipticCurve/Q/392/a/1}{392.a1}. It corresponds to the modular form $f_4$ with LMFDB label \href{https://www.lmfdb.org/ModularForm/GL2/Q/holomorphic/392/2/a/a/}{392.2.a.a}.
    \item $E_3 : y^2=x^3-x^2-16x+29$ be the elliptic curve with LMFDB label \href{https://www.lmfdb.org/EllipticCurve/Q/392/c/1}{392.c1}. It corresponds to the modular form $f_7$ with LMFDB label \href{https://www.lmfdb.org/ModularForm/GL2/Q/holomorphic/392/2/a/c/}{392.2.a.c}.
\end{itemize}

\begin{proposition}\label{prop: solutions Fermat and rational points}
    For every $(\star)$-solution $(a,b,c)$ of the equation $a^2+28b^3=27c^7$ (respectively, for every $(\star)$-solution $(a,b,c)$ of the equation $a^2+196b^3=27c^7$), at least one of the following holds:
    \begin{enumerate}
        \item $c=\pm 1$;
        \item there is a quadratic twist $E^{\chi}$ of $E_{(a,b,c)}$ (respectively, a quadratic twist $F^{\chi}$ of $F_{(a,b,c)}$) and an index $i \in \{1,2\}$ such that $E^\chi[7] \cong E_i[7]$ as Galois representations (respectively, $F^{\chi}[7] \cong E_i[7]$ for some $i \in \{1,2,3\}$). In particular, $j(E_{(a,b,c)})$ lies in the set
        \[
        \{ j(E) : [E] \in X_{E_i}^{\varepsilon}(\mathbb{Q}) \bigm\vert i \in \{1,2\}, \varepsilon \in \{+, -\} \},
        \]
        and $j(F_{(a,b,c)})$ lies in
        \[
        \{ j(E) : [E] \in X_{E_i}^{\varepsilon}(\mathbb{Q}) \bigm\vert i \in \{1,2,3\}, \; \varepsilon \in \{+, -\} \}.
        \]
    \end{enumerate}
\end{proposition}
\begin{proof}
Consider first the case of \Cref{eq: Cns} and let $E := E_{(a,b,c)}$. By \Cref{prop: level lowering} there are an integer $N \in \{2^2 \cdot 7^2, 2^3 \cdot 7^2, 2^4 \cdot 7^2\}$, a weight $2$ eigenform $f$ of level $\Gamma_0(N)$, and a prime $\mathfrak{p}$ of the ring $\mathbb{Z}[a_n(f)]$ of norm $7$ such that $\rho_{E,7} \sim \rho_{f, \mathfrak{p}}$. By Propositions \ref{prop: forms of level 196}, \ref{prop: forms of level 392} and \ref{prop: forms of level 784} we know that $f=f_i$ for some $i \in \{1, \ldots, 25\}$. By the same propositions, up to isomorphism of representations, we can replace $f_3, f_6, f_9, f_{17}, f_{20}, f_{25}$ with a different eigenform (at the same level) with rational coefficients, and therefore assume $i \not \in \{3, 6, 9, 17, 20, 25\}$.
\begin{enumerate}
    \item If $i=1$ or $2$, by \Cref{prop: forms of level 196} the representation $\rho_{E, 7}$ is a quadratic twist of $\rho_{f_1,7}$. In particular, there exists a quadratic twist $E^\chi$ of $E$ such that $\rho_{E^\chi,7} \sim \rho_{f_1, 7}$. By a similar argument using \Cref{prop: forms of level 392}, if $i=4$ or $5$ there is a quadratic twist $E^\chi$ of $E$ such that $\rho_{E^\chi, 7} \sim \rho_{f_4, 7}$.
    \item The cases $i=7, 8, 23, 24$ are impossible by \Cref{prop: 392 2 a c}.
    \item The cases $i=10, 11, 12, 21, 22$ are impossible by \Cref{prop: 392 2 a b and d}.
    \item If $c \neq \pm 1$, the case $i=13$ is impossible by \Cref{prop: 784.2.a.l}.
    \item The case $i=14$ is impossible by \Cref{prop: 784.2.a.f}.
    \item In cases $i= 15, 16$ (resp.~$i=18, 19$), \Cref{prop: forms of level 784} shows that a quadratic twist of $\rho_{f_i, 7}$ is isomorphic to $\rho_{f_1, 7}$ (resp.~$\rho_{f_4, 7}$), hence there is a quadratic twist $E^\chi$ of $E$ such that $\rho_{E^\chi, 7} \sim \rho_{f_1, 7}$ (resp.~$\rho_{E^\chi, 7} \sim \rho_{f_4, 7}$).
\end{enumerate}
This shows that at least one among (1) and (2) holds. The final statement of (2) follows from \Cref{lemma: why we use XE7}. We analyse $(\star)$-solutions of \Cref{eq: C sharp} in a similar way. Note that in the case of $F_{(a,b,c)}$ we cannot rule out the modular forms $f_{15}, f_{16}, f_{18}, f_{19}$, which means that in principle $F$ could have a quadratic twist $F^\chi$ with $\rho_{F^\chi, 7} \sim \rho_{f_1, 7}$ or $\rho_{F^\chi, 7} \sim \rho_{f_4, 7}$.
\end{proof}

We are thus reduced to computing the rational points on the six curves $X_{E_1}, X_{E_1}^-, X_{E_2}, X_{E_2}^-, X_{E_3}, X_{E_3}^-$. Two of these have no rational points for local reasons:
\begin{proposition}\label{prop: no local points}
    We have $X_{E_2}^-(\mathbb{Q})=X_{E_3}^-(\mathbb{Q})=\emptyset$.
\end{proposition}
\begin{proof}
Let $I_2$ be an inertia group at $2$.
Using results of Kraus \cite{KrausDefautSemistabilite} (see also \cite{CoppolaWild2adic}), we find that $\rho_{E_2, 7}(I_2) \cong \operatorname{SL}_2(\mathbb{F}_3)$. It follows from \cite[Theorem 4]{MR3493372} that there is no anti-symplectic isomorphism $E[7] \cong E_2[7]$ for any elliptic curve $E$ over $\mathbb{Q}$. Since $X_{E_2}^-$ parametrises anti-symplectic isomorphisms $E[7] \cong E_2[7]$, we conclude that $X_{E_2}^-(\mathbb{Q})=\emptyset$. The same exact argument applies to $X_{E_3}^-$.
\end{proof}
\begin{remark}
    Let $p$ be a prime and $E/\mathbb{Q}$ be an elliptic curve.
    The local obstructions at $p$ to the existence of an anti-symplectic isomorphism $E[7] \cong E'[7]$ correspond to the non-existence of $\mathbb{Q}_p$-points on the modular curve $X_E^-$. The previous proposition exploits such an obstruction at $p=2$, and indeed, it can also be proved computationally by checking that $X_{E_2}^-$ and $X_{E_3}^-$ have no $\mathbb{Q}_2$-points.
\end{remark}

\begin{remark}
    We note that $E_1$ is 3-isogenous to the elliptic curve $E_4$ with LMFDB label \href{https://www.lmfdb.org/EllipticCurve/Q/196/a/1}{196.a1}. Since $3$ is a non-square modulo 7, this isogeny gives an anti-symplectic isomorphism $\varphi : E_1[7] \cong E_4[7]$. Thus, given an  elliptic curve $E' / \mathbb{Q}$ and an anti-symplectic isomorphism $\psi : E'[7] \to E_1[7]$, by composing with $\varphi$ we obtain a symplectic isomorphism $\varphi\psi : E'[7] \cong E_4[7]$. In particular, $X_{E_1}^-$ is isomorphic over $\mathbb{Q}$ to $X_{E_4}$. This can also be checked directly using the formulas in the next proposition.
\end{remark}

\begin{proposition}
    Let $E$ be an elliptic curve with equation $y^2 = x^3 + ax +b$. An explicit model for the curve $X_E(7)$ is given by
    $$ax^4 + 7bx^3z + 3x^2y^2 - 3a^2x^2z^2 - 6bxyz^2 - 5abxz^3 + 2y^3z + 3ay^2z^2 + 2a^2yz^3 - 4b^2z^4 = 0.$$
    An explicit model for the curve $X_E^-(7)$ is given by
    \begin{gather*}
        -a^2x^4 + a(3a^3 +19b^2)y^4 + 3z^4 + 6a^2y^2z^2 + 6az^2x^2 - 6(a^3 +6b^2)x^2y^2 - 12aby^2zx \\ + 18bz^2xy + 2abx^3y - 12bx^3z - 2(4a^3 + 21b^2)y^3z + 2a^2by^3x - 8az^3y = 0.
    \end{gather*}
\end{proposition}

\begin{proof}
See \cite[Th\'eor\`eme 2.1]{kraus03} and \cite[equation (7.6)]{PSS07}.
\end{proof}

\begin{corollary}
    Equations for $X_E(7)$ for $E=E_1, E_2, E_3, E_4$ are
    \begin{align*}
        X_{E_1}(7) &: -x^4 - 4x^3y + x^3z + 3x^2y^2 - 12x^2yz - 12x^2z^2 + 6xy^2z - 6xyz^2 - 8xz^3 - 2y^3z - 4yz^3 + 4z^4 = 0 \\
        X_{E_2}(7) &: 4x^4 - 2x^3y - 12x^3z + 3x^2z^2 - 2xy^3 - 6xyz^2 + 7xz^3 + 2y^3z + 3y^2z^2 + 2yz^3 - 6z^4 = 0 \\
        X_{E_3}(7) &: x^4 + 3x^3y - 3x^2yz - 3x^2z^2 + 6xy^3 - 6xy^2z + 3xyz^2 - 2xz^3 + 4y^4 + 2y^3z - 5yz^3 = 0 \\
        X_{E_4}(7) &: -12x^4 - 6x^3y - 7x^3z + 3x^2y^2 + 6x^2yz - 6x^2z^2 + 4xy^3 - 6xy^2z + 3xz^3 + 2y^3z - 3y^2z^2 + 2yz^3 = 0,
    \end{align*}
    where $X_{E_4}(7) \cong X_{E_1}^-(7)$.
\end{corollary}

\section{Rational points on the curves $X_{E_i}$}\label{sect: genus 3}

As explained at the end of the previous section, we have reduced our problem to the following statement.

\begin{theorem}\label{thm: rational points}
    The rational points on $X_{E_1}, X_{E_2}, X_{E_4}$ are
    \begin{align*}
        X_{E_1}(\mathbb{Q}) &= \{ [0 : 1 : 0]\} \\
        X_{E_2}(\mathbb{Q}) &= \{ [0 : 1 : 0], \, [1 : 1 : 0]\} \\
        X_{E_4}(\mathbb{Q}) &= \{ [0: 1 : 0], \, [0 : 0 : 1] \}.
    \end{align*}
\end{theorem}

The proof strategy of \Cref{thm: rational points} is the same as \cite[\S 11 and 12]{PSS07}. We thus only quickly report the results of the various stages of the computation.

\subsection{Known rational points and proof of \Cref{thm: solutions twisted Fermat}}

For $i \in \{1,2,3,4\}$, define the following subsets of $X_{E_i}(\Q)$:
\begin{gather*}
    X_{E_1}(\Q)_{\operatorname{known}} = \{[0 : 1 :0]\}, \quad X_{E_2}(\Q)_{\operatorname{known}} = \{ [0 : 1 : 0], [1 : 1 : 0]\}, \quad X_{E_4}(\Q)_{\operatorname{known}}=\{[0:1:0], [0:0:1]\}, \\
    X_{E_3}(\Q)_{\operatorname{known}} = \{ [0 : 0 : 1], [1 : 1 : 1], [2 : 0 : 1], [-1 : 0 : 1] \}.
\end{gather*}
It is easy to check that $X_{E_i}(\Q)_{\operatorname{known}} \subseteq X_{E_i}(\Q)$, and
\Cref{thm: rational points} states that equality holds for $i \in \{1,2,4\}$. Rational points on $X_{E_i}$ correspond to certain elliptic curves over $\Q$, and \cite[\S 7.3]{PSS07} explains how to compute the $j$-invariant of the elliptic curve associated with a rational point on $X_{E_i}$. We find
\begin{gather*}
    \left\{ j(Q) : Q \in X_1(\Q)_{\operatorname{known}}\right\} = \{2^8 \cdot 7\}, \qquad 
    \left\{ j(Q) : Q \in X_2(\Q)_{\operatorname{known}}\right\} = \left\{2^8 \cdot 3^3 \cdot 7, \;\; \frac{2 \cdot 7 \cdot 3593^3}{(3^2 \cdot 5)^7}\right\}, \\
    \{ j(Q) : Q \in X_4(\Q)_{\operatorname{known}}\} = \left\{2^8 \cdot 7 \cdot 61^3, \;\; - \frac{2^4 \cdot 7 \cdot 1867^3}{17^7}\right\}, \\
    \{ j(Q) : Q \in X_3(\Q)_{\operatorname{known}}\} = \left\{ \frac{2^2 \cdot 253447^3}{67^7}, \;\; - \frac{2^{10} \cdot 19^3 \cdot 67^3 \cdot 2131^3}{317^7}, \;\; \frac{2^{10} \cdot 23^3}{5^7}, \;\; 2^8 \cdot 7^2\right\}.
\end{gather*}
Let $J_{\operatorname{known}}$ be the union of the sets corresponding to $X_1, X_2, X_4$.
\begin{proof}[Proof of \Cref{thm: solutions twisted Fermat} assuming \Cref{thm: rational points}]
    Combining \Cref{prop: solutions Fermat and rational points}, \Cref{prop: no local points}, \Cref{lemma: basic properties Eabc Fabc}, \Cref{thm: rational points}, and the above calculation of $j$-invariants, we find that if $(a,b,c)$ is a $(\star)$-solution of \Cref{eq: Cns}, then either $c=\pm 1$ or
    $
    j(E_{(a,b,c)}) =2^8 \cdot 7 \cdot \frac{b^3}{c^7}
    $
    lies in $J_{\operatorname{known}}$.
    
    If $c= \pm 1$, we compute the integral points on the elliptic curves $a^2+28b^3 = \pm 27$ and $a^2+196b^3 = \pm 27$ and find the solutions given in the statement of \Cref{thm: solutions twisted Fermat} (see also \Cref{rmk: small solutions}).
    
    Suppose instead that $j(E_{(a,b,c)})$ belongs to $J_{\operatorname{known}}$.    
    Since $(c, 2^8 \cdot 7 \cdot b^3)=1$ by definition of a $(\star)$-solution, we obtain that $\pm c^7$ is the denominator of one of the rational numbers in $J_{\operatorname{known}}$, and $\pm 2^8 \cdot 7 \cdot b^3$ is its numerator. This gives the possibilities $(b,c) \in (-1, - 1), (1,1), (-3, -1), (3,1), (-61, -1), (61, 1)$. It is straightforward to check that these pairs lead precisely to the solutions listed in the statement of \Cref{thm: solutions twisted Fermat}. 

    If we further assume \Cref{conj: XE3}, namely the equality $X_{E_3}(\Q)=X_{E_3}(\Q)_{\operatorname{known}}$, then a very similar argument solves \Cref{eq: C sharp}.
\end{proof}

\begin{remark}
    The $j$-invariants $\frac{2 \cdot 7 \cdot 3593^3}{(3^2 \cdot 5)^7}$ and $- \frac{2^4 \cdot 7 \cdot 1867^3}{17^7}$ do not correspond to $(\star)$-solutions of \Cref{eq: Cns}. However, they do correspond to the solutions
    \[
    (a,b,c) = (\pm 2\cdot 181 \cdot 313 \cdot 317, \; 3593, \; 90) \qquad \text{and} \qquad (a,b,c) = (\pm 2^{13} \cdot 5 \cdot 59957, \; -2^8 \cdot 1867, \; 2^4 \cdot 17)
    \]
    respectively, see also \Cref{rmk: small solutions}.
\end{remark}

\subsection{Two-descent and generators for finite-index subgroups of $J_i(\Q)$}

Let $i \in \{1, 2, 4\}$ and $X:=X_{E_i}$. We fix a known rational point $P_0$ on $X$: we take $P_0=[0:1:0]$ in all cases. For $i=2, 4$ we also let $P_1$ be the second known rational point on $X$, as listed in the statement of \Cref{thm: rational points}.

We let $J/\Q$ be the Jacobian of $X$. Our objective in this section is to compute the structure of $J(\Q)$ by determining its torsion and rank, and finding explicit generators for a finite-index subgroup.
As for the torsion subgroup, it is well known that $J(\Q)_{\operatorname{tors}}$ injects by reduction in $J(\F_p)$ for any prime $p>2$ of good reduction (see for example the appendix of \cite{MR604840}). Using $\#J(\F_p)=L(X_{\mathbb{F}_p},1)$, where $L(X_{\mathbb{F}_p},t)$ denotes the $L$-polynomial of the reduction of $X$ modulo $p$, we compute that $J(\Q)_{\operatorname{tors}}$ is trivial by showing that for a suitable, small set of primes $S_0$ we have $\operatorname{gcd} \left\{ \#J(\F_p) : {p \in S_0} \right\} = 1$. 

Next we briefly outline how we compute the rank of $J(\Q)$ and generators for a finite-index subgroup. We refer the reader to \cite[\S 11]{PSS07} for further details.
We construct a number field $K$ such that the base-change $X_{K}$ admits three flexes $T_1, T_2, T_3$ defined over $K$. These flexes have the property that the tangent line to $X$ at $T_j$ intersects $X$ at $T_j$ with multiplicity $3$ and at $T_{j+1}$ with multiplicity $1$ (indices read mod $3$). We construct a plane cubic $G$ such that $\{G=0\} \cap X = 2(T_1)+2(T_2)+2(T_3)+3(P_0)+R$, where $R$ is a $\Q$-rational divisor of degree $3$.
The coefficients of $G$ lie in a degree-8 étale $\Q$-algebra, which in every case turns out to be a degree-8 number field with trivial class group. We denote this number field by $A$. In the language of \cite[\S 11.1]{PSS07}, the fact that $A$ is a number field (equivalently, that its spectrum is connected) implies that the $\mathcal{G}_{\Q}$-set $T$ has only one orbit, of size $8$.
We work in the quotient $A^\times /  A^{\times 2}\Q^\times$ and define
\[
H := \left\{ a \in A^\times /  A^{\times 2}\Q^\times : \operatorname{Norm}_{A/\Q}(a) \in \Q^{\times 2}\right\}.
\]
Let $[U : V : W]$ be homogeneous coordinates on $\mathbb{P}^2$. By \cite[Propositions 11.4 and 11.5]{PSS07}, the function $g := \frac{G}{W^3}$ induces a group homomorphism
\begin{equation}\label{eq: descent map}
    g : \frac{J(\Q)}{2J(\Q)} \to H
\end{equation}
with kernel generated by $[R-3P_0]$. The element $[R-3P_0]$ is non-trivial in $J(\Q)/2J(\Q)$ since the $\mathcal{G}_{\Q}$-set $T$ has only orbits of even size. The image of the homomorphism \eqref{eq: descent map} lands in the \textit{fake $2$-Selmer group} $\operatorname{Sel}^2_{\operatorname{fake}}(J, \Q)$. We check that $X$ has good reduction outside $\{2, 7\}$ (this follows from the fact that $X$ is a twist of a curve with good reduction away from $7$ by a cocycle that is unramified outside $\{2,7\}$) and the Tamagawa number at $7$ is $1$. The set of primes $S$ of \cite[\S 11.1]{PSS07} can then be taken to be $\{2, \infty\}$. The group $\operatorname{Sel}^2_{\operatorname{fake}}(J, \Q)$ is explicitly computable. Imposing the local conditions at $2$ (but not at $\infty$) we find
\[
\dim_{\F_2}\operatorname{Sel}^2_{\operatorname{fake}}(J, \Q) \leq \begin{cases}
    0, \; \text{ if }i=1 \\
    1, \; \text{ if } i=2, 4.
\end{cases}
\]
Since the descent map \eqref{eq: descent map} has $1$-dimensional kernel, this shows that $\dim_{\F_2} \frac{J(\Q)}{2J(\Q)} \leq \begin{cases}
    1, \; \text{ if }i=1 \\
    2, \; \text{ if } i=2, 4.
\end{cases}$ As we already know that $J(\Q)$ has trivial torsion, we obtain the same bound for the rank of $J(\Q)$. For $i=1$, since $[R-3P_0]$ is a non-trivial element in $J(\Q)/2J(\Q)$, we obtain that $\operatorname{rk} J(\Q)=1$ and $[R-3P_0]$ is a generator of $J(\Q)$ up to finite index. For $i=2, 4$, we check that the difference of the two known rational points $[P_1-P_0]$ maps to the non-trivial element of $\#\operatorname{Sel}^2_{\operatorname{fake}}(J, \Q)$. This implies that the rank of $J(\Q)$ is $2$ and $[R-3P_0], [P_1-P_0]$ generate a finite-index subgroup.

\begin{remark}\label{rmk: rank of Jac XE3}
    Although not necessary for the proof of \Cref{thm: rational points}, we point out that we can carry out the same procedure also for $i=3$. We obtain $\operatorname{rk} J(\Q) =3$.
    Letting 
    \[
    P_0=[0,0,1], \; P_1=[1,1,1], \; P_2 =[2,0,1],\; P_3=[-1,0,1]
    \]
    be the four known rational points on $X_{E_3}$, we check that $D_1 := (P_1)-(P_0), D_2 := (P_2)-(P_0), D_3 := (P_3)-(P_0)$ generate a rank-3 subgroup of $J(\Q)$. More precisely, let $[R-3P_0]$ be as above and $D_4 := D_{4,+}-2(P_0)$, where $D_{4,+}$ is defined by $W=U^2-UV+2V^2=0$. The divisor classes $[D_1], [D_4]$ map to a basis of the 2-dimensional $\F_2$-vector space $\operatorname{Sel}^2_{\operatorname{fake}}(J, \Q)$, so $[R-3P_0], [D_1], [D_4]$ generate a finite-index subgroup of $J(\Q)$ and their classes form a basis of $J(\Q)/2J(\Q)$. We then compute
    \[
    2[R-3P_0] = 3[D_2] + 6[D_3], \quad 2[D_4] = -3 [D_1]+3[D_3] \quad \text{in }\,J(\Q),
    \]
    which shows that $[D_1], [D_2], [D_3]$ also generate a finite-index subgroup of $J(\Q)$.
\end{remark}

\subsection{Mordell-Weil sieve and saturation}

Since we have proved that for $i = 1,2,4$ the rank of the Mordell-Weil group $\operatorname{Jac}(X_{E_i})(\Q)$ is strictly smaller than the genus of $X_{E_i}$, we can in principle apply the Chabauty-Coleman method to these curves. For a suitably chosen prime $p$, we will apply this method to show that each $p$-adic disc centered at a point in $X(\Q)_{\operatorname{known}}$ contains at most one rational point -- necessarily the known one. This, however, does not yet suffice to determine $X(\Q)$: we must still check that the other $p$-adic discs contain no rational points at all. Equivalently, we need to prove that $X(\Q)$ and $X(\Q)_{\operatorname{known}}$ have the same reduction modulo $p$.
In order to achieve this, we apply the Mordell-Weil sieve.

The algorithm proceeds as follows. We let $Q_1, \dots , Q_r$ be generators of the Mordell-Weil group $J(\Q)$ and fix a rational point $P_0$ (we take $P_0=[0:1:0]$ in all cases, as in the previous paragraph). We implicitly extend $X$ and $J$ to an open subscheme of $\operatorname{Spec} \Z$, so that we can consider their reductions modulo good primes. By applying the Abel–Jacobi map based at $P_0$, we identify $X(\Q)$ with a subset of $J(\Q)$, and for each prime of good reduction $p$, we similarly view $X(\F_p)$ as a subset of $J(\F_p)$. Let $\pi_p$ be the reduction map from $J(\Q)$ to $J_{\F_p}(\F_p)$. Let $\operatorname{MW}_p$ be the subgroup of $J_{\F_p}(\F_p)$ generated by $\pi_p(Q_1), \dots, \pi_p(Q_r)$, and let $\Omega_p$ be the intersection $\operatorname{MW}_p \cap X_{\F_p}(\F_p)$. It is clear that the image of $X(\Q)$ through $\pi_p$ is contained in $\Omega_p$.
If we choose another prime $q \ne p$, we can further reduce the set $\Omega_p$ by using the fact that, for every element $a_1Q_1 + \ldots + a_rQ_r$, the projections $a_1\pi_p(Q_1) + \ldots + a_r\pi_p(Q_r)$ and $a_1\pi_q(Q_1) + \ldots + a_r\pi_q(Q_r)$  belong to $X_{\F_p}(\F_p)$ and $X_{\F_q}(\F_q)$ respectively. We then obtain that $\pi_p(X(\Q))$ is contained in the set $\Omega_p' := \{\pi_p(Q) \in \Omega_p \mid Q \in J(\Q), \ \pi_q(Q) \in \Omega_q \} \subset \Omega_p$.

\[\begin{tikzcd}
	{X(\Q)} && {X_{\F_p}(\F_p) \times X_{\F_q}(\F_q)} \\
	\\
	{J(\Q)} && {J_{\F_p}(\F_p) \times J_{\F_q}(\F_q)}
	\arrow["{\pi_p \times \pi_q}", from=1-1, to=1-3]
	\arrow[hook, from=1-1, to=3-1]
	\arrow[hook, from=1-3, to=3-3]
	\arrow["{\pi_p \times \pi_q}", from=3-1, to=3-3]
\end{tikzcd}\]
More generally, given a finite set $S$ of auxiliary primes of good reduction, we have
\[
\pi_p(X(\Q)) \subseteq \Omega'_{p, S} :=\{ \pi_p(Q) \in \Omega_p \bigm\vert Q \in J(\Q) \text{ such that } \pi_q(Q) \in \Omega_q \; \forall q \in S \}.
\]
Concretely, computing $\Omega'_{p, S}$ reduces to solving systems of congruences. Writing $Q = \sum_{i=1}^r a_i Q_i$, the condition $\pi_q(Q) \in \Omega_q$ imposes congruence conditions on the coefficients $a_i$ for each prime $q$, and we intersect these conditions as $q$ varies over $S \cup \{p\}$.
After fixing $p$, we look for a set $S$ of auxiliary primes such that $\Omega'_{p, S} = \pi_p(X(\Q)_{\operatorname{known}})$.

Unfortunately, in our case we do not know a set of generators of $J(\Q)$, but we know that $[R-3P_0]$ and $[P_1-P_0]$ generate a finite-index subgroup of $J(\Q)$ (generated only by $[R-3P_0]$ in the case $i=1$). Luckily, to run the above algorithm we simply need to know a subgroup $Z$ of $J(\Q)$ such that $\left(\prod_{q \in S \cup \{p\}}\pi_q \right)(Z)= \left(\prod_{q \in S \cup \{p\}}\pi_q \right)(J(\Q))$. The following lemma gives a sufficient condition for this equality to hold.
\begin{lemma}\label{lemma: sufficient condition projection}
    Let $A,B$ be two finite groups and let $G$ be a subgroup of $A \times B$. If $\pi_A, \pi_B$ are the projections onto $A$ and $B$ respectively, we have that
    $$[A \times B : G] \quad \text{divides} \quad [A : \pi_A(G)] \cdot [B : \pi_B(G)] \cdot \gcd(|\pi_A(G)|, |\pi_B(G)|).$$
\end{lemma}
\begin{proof}
    Since $G$ is a subgroup of $\pi_A(G) \times \pi_B(G)$, we have 
    $$[A \times B : G] = [A \times B : \pi_A(G) \times \pi_B(G)] \cdot [\pi_A(G) \times \pi_B(G) : G],$$
    so we can assume that $\pi_A(G)=A$ and $\pi_B(G) = B$, and prove that $[A \times B : G]$ divides $\gcd(|A|, |B|)$.
    Letting $N_A := \ker \pi_B|_G$ and $N_B := \ker \pi_A|_G$, we can identify $N_A$ and $N_B$ with normal subgroups of $A$ and $B$, respectively. We have that $N_A \times N_B \subseteq G$, and by Goursat's lemma we know that the image of $G$ in $\frac{A}{N_A} \times \frac{B}{N_B}$ is the graph of an isomorphism $\frac{A}{N_A} \cong \frac{B}{N_B}$. In particular, we have $\left| \frac{A}{N_A} \right| = \left| \frac{B}{N_B} \right| = \left| \frac{G}{N_A \times N_B} \right|$, and so
    $$[A \times B : G] = \left[ \frac{A}{N_A} \times \frac{B}{N_B} : \frac{G}{N_A \times N_B} \right] = \left| \frac{A}{N_A} \right| = \left| \frac{B}{N_B} \right|.$$
    This implies that $[A \times B : G]$ must divide both $|A|$ and $|B|$, concluding the proof.
\end{proof}
We say that $Z$ is \textit{saturated} at $\ell$ if $\ell \nmid [J(\Q) : Z]$. By \Cref{lemma: sufficient condition projection}, in order to run the Mordell-Weil sieve for fixed $Z$, $p$ and $S$, we only need to know saturation at a finite set of primes, namely the divisors of $\prod_{q \in S \cup \{p\}} [J(\F_q) : \pi_q(Z)]$ and of $\gcd(|\pi_q(Z)|, |\pi_{q'}(Z)|)$ for every pair of distinct primes $q,q'$ in $S \cup \{p\}$. We now explain how to check that a subgroup $Z$ (for which we know explicit generators) is saturated at a given prime $\ell$. Consider the set
\[
Y_{\ell} := \left\{ \sum_{i=1}^r a_i Q_i \bigm\vert (a_1, \ldots, a_r) \in \{0, \ldots, \ell-1\}^r \setminus \{ (0, \ldots, 0) \} \right\}.
\]
Assume that for each $y \in Y_{\ell}$ there is a prime $s$ (of good reduction) such that $\pi_s(y)$ is not divisible by $\ell$ in the finite abelian group $J(\F_s)$. We claim that $Z$ is then saturated at $\ell$. Indeed, aiming for a contradiction, suppose that $Z$ is not saturated. Then, there exists an element $D \in J(\Q)$ such that $D \not \in Z$ but $\ell D \in Z$. Write $\ell D = \sum_{i=1}^r (\ell b_i+a_i) Q_i$, where the $b_i$ are integers and the $a_i$ are in $\{0, \ldots, \ell-1\}$. As $D$ is not in $Z$, not all $a_i$ are $0$ (here we use the fact that $J(\Q)[\ell]=\{0\}$, which we have checked in our case). In particular, the element $y := \sum_{i=1}^r a_i Q_i$ is in $Y_\ell$ and is divisible by $\ell$ in $J(\Q)$, since $y = \ell \left(D - \sum_{i=1}^r b_iQ_i\right)$. By assumption, there is some good prime $s$ for which $\pi_s(y)$ is not in $\ell J(\F_s)$. On the other hand we have
\[
\pi_s(y) = \ell \pi_s\left(D - \sum_{i=1}^r b_iQ_i\right) \in \ell J(\F_s):
\]
this contradicts the assumption, which shows that $Z$ is saturated at $\ell$, as claimed.

Thus, our algorithm to check for saturation proceeds as follows: for each $y \in Y_{\ell}$, we look for an auxiliary prime $s_{y}$ such that $\pi_{s_y}(y)$ does not lie in $\ell J(\F_{s_y})$. Finding such a prime for every $y \in Y_\ell$ confirms saturation at $\ell$. Since only finitely many primes $\ell$ are involved (and, in practice, just a few small ones), the overall task is computationally manageable.

\begin{remark}
We add some details about the practical implementation of these calculations.
\begin{enumerate}
    \item Simple arithmetic considerations show that we can restrict the set $Y_\ell$ somewhat -- by rescaling, it suffices to consider vectors $(a_1, \ldots, a_r)$ whose first non-zero entry is equal to $1$. This speeds up the computation considerably.
    \item Very often, the same prime $s_y$ works for most (if not all) elements $y \in Y_\ell$.
    \item In fact, the classes $[R-3P_0], [P_1-P_0]$ described above do not generate a saturated subgroup of $J(\Q)$ (saturation at $3$ fails for each $i=1, 2, 4$). In our computations, we find generators of a slightly larger subgroup and use this as our $Z$. We strongly suspect that this larger subgroup is all of $J(\Q)$. In any case, it was easy to test that $Z$ is saturated at all the small primes we need.
\end{enumerate}
\end{remark}

In our concrete cases, we always find a suitable prime $p$ and a set $S$ such that the above algorithms (Mordell-Weil sieve and checking for saturation) prove $\Omega_{p,S}' = \pi_p(X(\Q)_{\operatorname{known}})$. Specifically, we have the following.
\begin{itemize}
    \item For $i=1$ we choose $p=41$ and $S=\{53,71\}$. The subgroup $Z$ is saturated at every prime dividing $\prod_{q \in S \cup \{p\}} [J(\F_q) : \pi_q(Z)]$ and every $\gcd(|\pi_q(Z)|, |\pi_{q'}(Z)|)$, and the Mordell-Weil sieve shows $\pi_p(X(\Q)) = \pi_p(X(\Q)_{\operatorname{known}})$.
    \item For $i=2$ we choose $p=73$. The subgroup $Z$ is saturated at every prime dividing $[J(\F_p) : \pi_p(Z)]$. In this case, we do not need to consider any other primes $q$, as the set $\Omega_p$ already coincides with $\pi_p(X(\Q)_{\operatorname{known}})$.
    \item For $i=4$ we choose $p=31$ and $S=\{3\}$. The subgroup $Z$ is saturated at every prime dividing $\prod_{q \in S \cup \{p\}} [J(\F_q) : \pi_q(Z)]$ and $\gcd(|\pi_3(Z)|, |\pi_{31}(Z)|)$, and again the Mordell-Weil sieve proves $\pi_p(X(\Q)) = \pi_p(X(\Q)_{\operatorname{known}})$.
\end{itemize}

\subsection{Determination of the rational points}

For $i=1, 2, 4$ we let respectively
\[
p=41, \quad p=73, \quad p=31.
\]
As explained above, using the Mordell-Weil sieve we have proven that for each rational point $Q \in X(\Q)$ there is a point $Q_1 \in X(\Q)_{\operatorname{known}}$ such that
\[
Q \equiv Q_1 \pmod{p}.
\]
Thus, it suffices to find the rational points in the $p$-adic discs around the known rational points $X(\Q)_{\operatorname{known}}$. We do this using the Chabauty-Coleman method.

In \cite{PSS07}, the authors find small multiples of their generators of (a finite-index subgroup of) $J(\Q)$ that give elements in the kernel of the reduction map $J(\Q) \to J(\F_p)$, which allows them to only consider Coleman integrals between points in a single residue disc. The primes $p$ we work with are somewhat larger than those used in \cite{PSS07}, so we find it easier to use recent advances in the computation of Coleman integrals to work over the whole curve (and not just a single residue disc). We now give some details.

Let $\mathcal{X} / \Z_p$ be a proper smooth model of $X/\Q_p$. Recall that we have explicit divisors $D_1$ or $D_1, D_2$ whose classes generate $J(\Q)$ up to finite index.
We compute a regular differential $\omega \in H^1(\mathcal{X},\Omega^1_{\mathcal{X}/\Z_p})$ such that
\[
\int_{[D_j]} \omega = 0
\]
for each $[D_j]$. By properties of the Coleman integral, this implies that $\int_{[D]} \omega =0$ for any $[D] \in J(\Q)$ (note that here it suffices to know that the $[D_j]$ generate a finite-index subgroup of $J(\Q)$). Fixing a basis $\omega_1, \omega_2, \omega_3$ of the regular differentials on $\mathcal{X}$ and writing $\omega=\sum_{k=1}^3 a_k\omega_k$, this is a linear system of at most two equations in the three variables $a_1, a_2, a_3$. All we need to write down these equations explicitly and solve them up to finite precision is the ability to compute $\int_{[D_j]} \omega_k$.

We perform this calculation using the implementation of Coleman integrals described in \cite{MR4136553}. We note that this implementation allows for the computation of Coleman integrals between points on the curve that are defined over $\Q_p$, but not over finite extensions. This is not a problem for the divisors $[P_1-P_0]$ (for $i=2, 4$). For the divisor $[R-3P_0]$, we proceed as follows:
\begin{enumerate}
    \item in cases $i=1, 2$, the divisor $R$ splits over $\Q_p$ as the sum of three $\Q_p$-rational points $R_1, R_2, R_3$. We simply integrate between $P_0$ and each $R_i$ and sum the results.
    \item in case $i=4$, the divisor $R$ does not split over $\Q_p$. Using Riemann-Roch we find an explicit function $f$ such that
    \[
    2R - 6P_0 = R' - 3P_0 + \operatorname{div} f
    \]
    for another explicit divisor $R'$ of degree $3$. We check that $R'$ splits as the sum of three $\Q_p$-rational points $R'_1, R'_2, R'_3$ and proceed as above, using the fact that
    \[
    \int_{[R-3P_0]} \omega = \frac{1}{2} \int_{[2R-6P_0]} \omega  = \frac{1}{2} \int_{[R'-3P_0]} \omega = \frac{1}{2} \sum_{j=1}^3 \int_{P_0}^{R'_j} \omega.
    \]
\end{enumerate}
Once an annihilating differential $\omega$ is known, we apply the Chabauty-Coleman strategy: we write down an analytic function on each residue disc whose zeroes contain the $\Q$-rational points of $X$ and, using \cite{MR4136553} again, we find in each case that these analytic functions have precisely one zero in the $p$-adic residue disc of each $Q \in X(\Q)_{\operatorname{known}}$, which must therefore be the known rational point. This concludes the proof of \Cref{thm: rational points}.

\subsection{The curve $X_{E_3}$}

We briefly describe some properties of the genus-$3$ curve $X_{E_3}$ and comment on possible strategies to determine its rational points.
\begin{enumerate}
    \item The curve $X_{E_3}$ has rational points, one of which corresponds to an actual solution to the Diophantine problem under consideration. This presumably prevents the use of local obstructions to exclude other rational points (see \cite{MR4770972} for a study of local obstructions to rational points on twists of the Klein quartic).
    \item As explained in \Cref{rmk: rank of Jac XE3}, $\operatorname{Jac}(X_{E_3})(\Q)$ has rank 3, so we cannot apply the Chabauty-Coleman method.
    \item The Mordell-Weil group $\operatorname{Jac}(X_{E_3})(\Q)$ has no torsion, making it difficult to construct étale covers of $X_{E_3}$ over $\Q$ (which could facilitate descent). We also attempted to construct an étale double cover over a number field; however, the field of definition of any bitangent has degree $28$ over $\Q$, which makes this approach computationally impractical.

    \item One can look for other torsion points of $\operatorname{Jac}(X_{E_3})$ defined over small number fields. For example, we tested that $\operatorname{Jac}(X_{E_3})(\F_p)$ has a $7$-torsion point for all primes $p \leq 500$ with $p \equiv \pm 1 \pmod{7}$. This suggests that $\operatorname{Jac}(X_{E_3})$ may have a $7$-torsion point over $\Q(\zeta_7)^+$. If true, this can probably be established rigorously by considering the mod-$7$ Galois representation of $\operatorname{Jac}(X(7))$ and the way the representation changes under twists, but it is not clear that knowing a $7$-torsion point over a cubic number field would help (the corresponding étale cover would have genus 15).
    
    \item The Jacobian $\operatorname{Jac}(X_{E_3})$ is geometrically isomorphic to the Jacobian of the Klein quartic, which splits (up to isogeny) as the cube of a CM elliptic curve. The minimal field of definition of the endomorphisms can be determined using \cite[Propositions 3.34 and 3.35]{MR3864839}; it is the splitting field over $\Q$ of the polynomial $x^8 - 2x^7 + 7x^4 - 14x^2 + 8x + 5$. This field has degree $336$, making computations over it completely impractical.
    
    \item No non-trivial endomorphism of $\operatorname{Jac}(X_{E_3})$ is defined over $\Q$. On the other hand, we have
    \[
    \operatorname{End}\left(\operatorname{Jac}(X_{E_3})_{\Q(\sqrt{-7})} \right) \otimes_{\Z} \Q \cong \Q(\sqrt{-7}).
    \]
    We show this statement, starting with the following general lemma (cf.~also \cite[Proposition 2.3]{MR1630512}).

    \begin{lemma}\label{lemma: centre under twists}
        Let $A_1, A_2$ be abelian varieties over $\Q$ that are isomorphic over $\overline{\Q}$. For $i=1, 2$, let $Z_i$ be the centre of the geometric endomorphism ring $\operatorname{End}(A_{i, \overline{\Q}})$. The action of $\operatorname{Gal}\left(\overline{\Q}/\Q \right)$ on $\operatorname{End}(A_{i, \overline{\Q}})$ stabilises $Z_i$; let $\psi_i : \operatorname{Gal}\left(\overline{\Q}/\Q \right) \to \operatorname{Aut}(\operatorname{End}(Z_i))$ be the induced map. The fixed field $F_i$ of $\ker \psi_i$ is the field of definition of the action of $Z_i$ on $A_i$. We have $F_1=F_2$.
    \end{lemma}
    \begin{proof}
        Fix a $\overline{\Q}$-isomorphism $\varphi : A_{1, \overline{\Q}} \to A_{2, \overline{\Q}}$. The choice of $\varphi$ induces an isomorphism
        \[
        \begin{array}{cccc}
            \Phi & : Z_1 & \to & Z_2 \\
            & \alpha & \mapsto & \varphi \circ \alpha \circ \varphi^{-1}.
        \end{array}
        \]
        The Galois action on $\Phi(\alpha)$ is given by
        \[
        \begin{aligned}
            {}^\sigma \Phi(\alpha) & = {}^\sigma \varphi \circ {}^\sigma \alpha \circ {}^\sigma \varphi^{-1} \\
            & = {}^\sigma \varphi \circ \varphi^{-1} \circ \varphi \circ {}^\sigma \alpha \circ \varphi^{-1} \circ \varphi \circ {}^\sigma \varphi^{-1} \\
            & = ( {}^\sigma \varphi \circ \varphi^{-1}) \circ \Phi({}^\sigma \alpha) \circ ( {}^\sigma \varphi \circ \varphi^{-1})^{-1};
        \end{aligned}
        \]
        since $\Phi({}^\sigma \alpha)$ is in the centre of the endomorphism ring of $A_{2, \overline{\Q}}$ and ${}^\sigma \varphi \circ \varphi^{-1}$ is an automorphism of $A_{2, \overline{\Q}}$, we obtain
        \[
        {}^\sigma \Phi(\alpha) = ( {}^\sigma \varphi \circ \varphi^{-1}) \circ \Phi({}^\sigma \alpha) \circ ( {}^\sigma \varphi \circ \varphi^{-1})^{-1} = \Phi({}^\sigma \alpha) \circ ( {}^\sigma \varphi \circ \varphi^{-1}) \circ  \circ ( {}^\sigma \varphi \circ \varphi^{-1})^{-1}= \Phi({}^\sigma \alpha),
        \]
        that is, $\Phi$ is Galois-equivariant. This implies that the subgroups of $\operatorname{Gal}(\overline{\Q}/\Q)$ fixing $\alpha$ and $\Phi(\alpha)$ are equal, for all $\alpha \in Z_1$. By taking the intersection of these subgroups over all $\alpha \in Z_1$ we obtain $\ker \psi_1 = \ker \psi_2$, that is, $F_1=F_2$ as claimed.
    \end{proof}
    We apply \Cref{lemma: centre under twists} to $A_1 = \operatorname{Jac}(X_{E_3})$ and $A_2 = \operatorname{Jac}(X(7))$. In the case of $A_2$, one knows that $\operatorname{Jac}(X(7))_{\overline{\Q}} \sim E^3$, where $E$ is an elliptic curve with CM by $\Q(\sqrt{-7})$, hence $\operatorname{End}\left( \operatorname{Jac}(X(7))_{\overline{\Q}}  \right) \otimes \Q \cong M_3(\Q(\sqrt{-7}))$.
    The centre of $\operatorname{Jac}(X(7))_{\overline{\Q}}$ is therefore an order in $\Q(\sqrt{-7})$, and it's not hard to show that its action is defined over $\Q(\sqrt{-7})$. From the lemma we obtain that $\operatorname{Jac}(X_{E_3})_{\Q(\sqrt{-7})}$ contains an order in $\Q(\sqrt{-7})$. On the other hand, by reducing modulo primes of $\Z[\frac{1+\sqrt{-7}}{2}]$ lying above 11 and 23 (both 11 and 23 are split in this ring) we conclude first that $\operatorname{Jac}(X_{E_3})_{\Q(\sqrt{-7})}$ is irreducible, and then that its endomorphism ring is an order $\mathcal{O}$ in $\Q(\sqrt{-7})$ -- see also \cite{MR3904148, MR4280568, MR3882288} for more details on using reductions to give upper bound on endomorphism rings. Since the action of $Z_1$ is defined over $\Q(\sqrt{-7})$ we must have $\mathcal{O} \subset Z_1$. On the other hand, since the action of $Z_2$ (hence of $Z_1$) is \textit{not} defined over $\Q$ we obtain that $\operatorname{End}\left(\operatorname{Jac}(X_{E_3})\right)$ is trivial as claimed.
\end{enumerate}

The rich endomorphism structure of $\operatorname{Jac}(X_{E_3})$ offers some hope that \Cref{conj: XE3} could be resolved using current or near-future techniques, such as the quadratic Chabauty method over number fields as in \cite{balakrishnan25}. 
\begin{remark}
    {The quadratic Chabauty method takes as input a \textit{Rosati-symmetric} endomorphism of $J$. Unfortunately, the restriction of the Rosati involution to the centre of $M_{3}(\Q(\sqrt{-7}))$ is complex conjugation, which means that the non-trivial endomorphism we found above is not Rosati-symmetric.}
\end{remark}
Finally, we note that, for our application to modular curves, it suffices to solve \Cref{eq: C sharp} in the special case where $c = -z^2$ for some integer $z$. Furthermore, the parameter $c$ is closely related to the denominator of the $j$-invariant of the elliptic curve $F_{(a,b,c)}$, and we have explicit formulas for the $j$-invariants of the elliptic curves corresponding to rational points on $X_{E_3}$. These additional constraints may help isolate the subset of $X_{E_3}(\Q)$ that is relevant to the determination of the rational points on $X_{ns}^\#(49)$ and $X_{sp}^\#(49)$.

\appendix\section*{Appendix}

In this appendix, we give a proof of \Cref{prop: cartan7p}. 

\begin{proof}[Proof of \Cref{prop: cartan7p}] 
    Let $p>7$ be a prime and let $E/\Q$ be an elliptic curve such that both $\operatorname{Im} \rho_{E, 7}$ and $\operatorname{Im} \rho_{E, p}$ are contained in the normaliser of a non-split Cartan.
    By \Cref{prop: cartan49 denominator} this implies that the denominator of the rational number $j(E)$ is a $7p$-th power. On the other hand, repeating the proof of \Cref{prop: norm equations} with $7p$ in place of $49$, we find that there exist coprime integers $x,y$ such that $j(E) = \frac{g(x,y)^3}{f(x,y)^7}$, where $f,g$ are homogeneous polynomials with integer coefficients and $f(x,y) = x^3 - 7x^2y + 7xy^2 + 7y^3$. It follows that there exist $k \in \{1,7,8,56\}$ and an integer $z$ such that $f(x,y) = kz^p$.
    By \Cref{cor: specific norm equations}, to every solution of $f(x,y)=kz^n$ we can attach a solution $(a,b,c) = (a,b,z^2)$ of one of the equations $a^2 + 28b^3 = -27c^p$ or $a^2 + 196b^3 = -27c^p$. Since $c=z^2$ is positive, this can be rewritten as
    \begin{equation*}
        a^2 + 27c^p = 28(-b)^3 \qquad \text{or} \qquad a^2 + 27c^p = 196(-b)^3,
    \end{equation*}
    where clearly $-b$ is positive.
    Since $p>4$, \Cref{cor: specific norm equations} also gives that the greatest common divisor $D$ of the numbers $a^2, 27c^p, 196(-b)^3$ divides $27$.
    Dividing through by $D$, we find the $abc$ triple
    \[
    \left( \frac{a^2}{D}, \frac{27 c^p}{D}, \frac{28 |b|^3}{D} \right) \quad \text{ or }\quad \left( \frac{a^2}{D}, \frac{27 c^p}{D}, \frac{196 |b|^3}{D} \right).
    \]    
    Fix $0<\varepsilon<\frac{1}{41}$. Recalling that $D \leq 27$, the $abc$ conjecture implies that there exists a constant $K$ (independent of $p, a, b, c$) such that
    \[
    \frac{28}{27} |b|^3 < K \operatorname{rad}\left( 2 \cdot3 \cdot 7 \cdot  |abc| \right)^{1+\varepsilon},
    \]
    hence
    \[
    |b|^3 < K' (|abc|)^{1+\varepsilon}
    \]
    for some absolute constant $K'$. Since $a^2 < 196|b|^3$ and $c^p < \frac{196}{27}|b|^3$,
    we obtain another absolute constant $K''$ such that
    \[
    |b|^3 < K'' |b|^{(1+\varepsilon)(1+3/2 + 3/p)} \le K'' |b|^{(1+\varepsilon)(1+3/2 + 3/7)},
    \]
    which -- given our choice of $\varepsilon$ -- implies that $|b|$ is bounded. The inequality $c^p < \frac{196}{27}|b|^3$ now implies that for $p \gg 0$ we must have $c \in \{0, 1\}$, for otherwise $c^p$ would be too large. Since the equation $f(x,y)=0$ has no solutions in coprime integers, we deduce that $c=1$, $z=\pm 1$, and therefore $f(x,y)=kz^p=\pm k$ belongs to a finite list of possibilities. For fixed $k$, these are Thue equations, each of which has finitely many solutions. This gives a finite number of possibilities for $(x,y)$ and hence for $j(E) = \frac{g(x,y)^3}{f(x,y)^7}$. Let $\{E_1, \ldots, E_m\}$ be a finite list of curves over $\Q$ realising these $j$-invariants. 
    
    Summarising, for $p$ large enough, the above argument shows that if $E$ has a non-split Cartan structure both at level $7$ and at level $p$, then $j(E) \in \{j(E_1), \ldots, j(E_m)\}$. However, each non-CM curve among $\{E_1, \ldots, E_m\}$ does not have a non-split Cartan structure for $p$ large enough (and this property is invariant under quadratic twist), so for $p$ large enough we obtain that $j(E)$ is one of the CM $j$-invariants in the list $\{j(E_1), \ldots, j(E_m)\}$.
\end{proof}

\bibliographystyle{abbrv}
\bibliography{bib}

\end{document}